\documentclass[12pt,a4paper,reqno]{amsart}
\usepackage{amsmath,amsthm,amsfonts,amssymb,bbm}
\usepackage{graphicx,psfrag,subfigure,color}
\usepackage{cite}
\usepackage{hyperref,enumerate}

\numberwithin{equation}{section}

\newcommand{\e}{\varepsilon}

\renewcommand{\Re}{\mathrm{Re}}
\renewcommand{\Im}{\mathrm{Im}}

\DeclareMathOperator{\Tr}{Tr}

\newtheorem{prop}{Proposition}[section]
\newtheorem{Theorem}[prop]{Theorem}
\newtheorem{Lemma}[prop]{Lemma}
\newtheorem{Definition}[prop]{Definition}

\newtheorem{Claim}[prop]{Claim}

\newtheorem{Remark}[prop]{Remark}

\title[AKPZ growth process with a smooth phase]{A $(2+1)$-dimensional Anisotropic KPZ growth model with a smooth phase}
\author{Sunil Chhita}
\address{Department of Mathematical Sciences, Durham University, Stockton Road, Durham, DH1 3LE, UK.} \email{sunil.chhita@durham.ac.uk} 
\author{Fabio Lucio Toninelli}
\address{Univ Lyon, CNRS, Université Claude Bernard Lyon 1,  UMR 5208, Institut Camille Jordan,  F-69622 Villeurbanne cedex, France} \email{toninelli@math.univ-lyon1.fr}

\begin{document}
\maketitle

\begin{abstract} Stochastic growth processes in dimension $(2+1)$ were
  conjectured by D. Wolf, on the basis of renormalization-group
  arguments, to fall into two distinct universality classes, according
  to whether the Hessian $H_\rho$ of the speed of growth $v(\rho)$ as
  a function of the average slope $\rho$ satisfies $\det H_\rho>0$
  (``isotropic KPZ class'') or $\det H_\rho\le 0$ (``anisotropic KPZ
  (AKPZ)'' class). The former is characterized by strictly positive
  growth and roughness exponents, while in the AKPZ class fluctuations
  are logarithmic in time and space.

  It is natural to ask (a) if one can exhibit interesting growth
  models with ``smooth'' stationary states, i.e., with $O(1)$
  fluctuations (instead of logarithmically or power-like growing, as
  in Wolf's picture) and (b) what new phenomena arise when $v(\cdot)$
  is not differentiable, so that $H_\rho$ is not defined. The two questions
  are actually related and here we provide an answer to both, in a
  specific framework. We define a $(2+1)$-dimensional interface growth
  process, based on the so-called shuffling algorithm for domino
  tilings. The stationary, non-reversible measures are
  translation-invariant Gibbs measures on perfect matchings of
  $\mathbb Z^2$, with $2$-periodic weights. If $\rho\ne0$, fluctuations are known to grow logarithmically in space
and to  behave like a two-dimensional GFF.
  We prove
  that fluctuations grow at most logarithmically in time and that $\det H_\rho<0$: the
  model belongs to the AKPZ class.  When $\rho=0$, instead, the
  stationary state is ``smooth'', with correlations uniformly bounded
  in space and time; correspondingly, $v(\cdot)$ is not differentiable
  at $\rho=0$ and we extract the singularity of the eigenvalues of
  $H_\rho$ for $\rho\sim 0$.

\end{abstract}
\section{Introduction}

A statistical physicist's view of stochastic interface growth models
is that they describe the diverse phenomena of interface growth 
and crystal deposition \cite{BS95}.  These models are related
mathematically to both interacting particle systems \cite{KL99,Spo91}
and stochastic partial differential equations, in particular, the
celebrated Kardar-Parisi-Zhang (KPZ) equation \cite{KPZ86}.

A $(2+1)$-dimensional stochastic growth process describes the time evolution,
in three-dimensional space, of a two-dimensional interface, modeled
mathematically as the graph of a function $h$ from $\mathbb Z^2$ (or
some other two-dimensional lattice) to $\mathbb R$.  The dynamics form
an irreversible Markov chain and the transition rates, by which the
interface height  increases or decreases, are asymmetric and
induce a non-trivial drift, that depends on the average local  slope. Among
the most interesting and challenging questions are those of
describing the stationary states, of obtaining the deterministic PDE
that describes on large scales the typical evolution of the height
function (hydrodynamic limit) and of understanding the large-scale space-time structure of the height fluctuations.

In many natural cases, given a slope $\rho\in \mathbb R^2$, there
exists a unique stationary state $\mu_\rho$ with average slope $\rho$,
i.e. $\mu_\rho(h(x)-h(y))=\rho\cdot(x-y)$. What is stationary is
actually only the law of the \emph{height gradients}
$\{h(x)-h(y)\}_{x,y\in\mathbb Z^d}$, while the average height itself
grows linearly in time, with a certain speed of growth $v(\rho)$. The
PDE describing the hydrodynamic limit is then expected to be of the
Hamilton-Jacobi type
\begin{eqnarray}
  \label{eq:HJ}
  \partial_t \psi=v(\nabla \psi)
\end{eqnarray}
with $\psi= \psi(x,t), t\ge0,x\in \mathbb R^d$ representing the suitably rescaled height profile, $\nabla$ denoting gradient w.r.t. the space variable and 
the solution of the PDE being interpreted in the vanishing viscosity sense.

As far as fluctuations are concerned, one usually introduces a roughness exponent $\alpha$ and a growth exponent $\beta$, that measure how height fluctuations grow in space and time in the stationary process, namely
\begin{eqnarray}
  \label{eq:alpha}
  {\rm Var}_{\mu_\rho}(h(x)-h(y))\stackrel{|x-y|\to\infty}\sim const.\times(1+ |x-y|^{2\alpha})
\end{eqnarray}
and
\begin{eqnarray}
  \label{eq:beta}
   {\rm Var}_{\mu_\rho}(h(x,t)-h(x,0))\stackrel{t\to\infty}\sim const.\times(1+t^{2\beta}),
\end{eqnarray}
vanishing exponents usually meaning logarithmic growth. (The exponents
can also be negative, in which case the constant $1$ dominates
asymptotically; this is the case for the stochastic heat equation with additive noise in dimension $d\ge3$, see below.)

Heuristically, one expects height fluctuations in the stationary process to
be qualitatively described, \emph{on large space-time scales}, by a
stochastic PDE (KPZ equation) of the type 
\begin{eqnarray}
  \label{eq:SPDE}
  \partial_t \phi(x,t)=\Delta \phi(x,t)+\langle\nabla \phi(x,t),H_{\rho} \nabla \phi(x,t)\rangle+\xi(x,t),
\end{eqnarray}
where the diffusive Laplacian term tends to locally smooth out
fluctuations, the $2\times 2$ symmetric matrix $H_\rho$ in the non-linear term is the Hessian
of the function $v(\cdot)$ computed at $\rho$ and $\xi(x,t)$ is a
space-time noise that contains the randomness of the process. Equation
\eqref{eq:SPDE} is singular if $\xi$ is space-time white
noise (Hairer's theory of regularity structures gives a meaning to it
in space dimension $d=1$ \cite{Hai11}, where $H_\rho$ reduces to the
scalar quantity $v''(\rho)$). On the other hand we are interested in
properties on large space-time scales and lattice models have a
natural ``ultraviolet'' space cut-off (lattice
spacing), so we can as well assume that the noise is smooth in space and  correlated over  distances of order $1$.

Most of the known rigorous results on stochastic growth models have
been proven in dimension $(1+1)$, often in cases where the stationary
measures of the interface gradients are of i.i.d. type.  It is then
expected (and mathematically proved in many examples, cf.  e.g.
\cite{FS10,Cor11,QS15} for recent reviews on the mathematical aspects
of the KPZ equation and universality class in one dimension) that when
$v''(\rho)\ne0$ the exponents $\alpha,\beta$ are given by
$\alpha_{d=1}=1/2,\beta_{d=1}=1/3$, as already predicted in the
original work \cite{KPZ86}. These exponents differ from those
($\alpha_{SHE}=(2-d)/2, \beta_{SHE}=(2-d)/4$) of the linear stochastic
heat equation with additive noise, obtained by dropping the non-linear
term in \eqref{eq:SPDE}.  On the other hand, for $(d+1)$-dimensional
models, $d\ge 3$, renormalization-group computations \cite{KPZ86,BS95}
applied to the stochastic PDE \eqref{eq:SPDE} predict that, if the
non-linear term is sufficiently small, then it is irrelevant, meaning
that the large-scale fluctuation properties and exponents
$\alpha,\beta$ are the same as those of the stochastic heat equation
(cf. \cite{magnen2017diffusive,gu2017edwards} for recent mathematical
progress in the $\lambda\ll1$ regime; let us also add that
\cite{KPZ86} predicts a transition at a critical non-linearity
$\lambda_c\ne0$ but nothing is known rigorously).

 In the $(2+1)$-dimensional case we consider here, the picture is
 different. On the basis of a renormalization-group analysis of
 \eqref{eq:SPDE} by D. Wolf \cite{Wol91}, of numerical simulations
 \cite{TFW92,HH13,halpin1992kinetic} and on a few mathematically treatable models (see
 references in point (ii) below), the following conjectural picture has emerged (see
 also the introduction of \cite{toninelli20173+} for a more detailed
 discussion):
 \begin{enumerate}
 \item [(i)] If $\det H_\rho$ is strictly positive, i.e. if
 $v(\cdot)$ is convex or concave with non-zero Gaussian curvature, the
 growth model is said to belong to the ``KPZ (or Isotropic KPZ)
 class''. In this case, the exponents $\alpha,\beta$ are strictly
 positive and, in particular, different from those of the
 two-dimensional stochastic heat equation, which are both zero.   The actual values of the exponents are known numerically to a high degree of precision \cite{TFW92,HH13}, but we are not aware of any rigorous or even convincing arguments to predict them.
\item [(ii)] If instead $\det H_\rho\le 0$ (``Anisotropic KPZ'' or
  AKPZ class) one expects that $\alpha=\beta=0$ and that moreover the
  growth of \eqref{eq:alpha}, \eqref{eq:beta} is logarithmic, exactly
  like for the two-dimensional stochastic heat equation. This belief
  is supported by the mathematical analysis of various
  $(2+1)$-dimensional growth models
  \cite{BF08,toninelli20172+,CF15,chhita2017speed} that share the
  following features: stationary states can be found explicitly and
  their height fluctuations behave on large space scales as a massless
  Gaussian Field (GFF), with $\alpha=0$ and logarithmic correlations;
  the speed of growth $v(\cdot)$ can be computed and $\det H_\rho$
  turns out to be negative; the growth exponent $\beta$ is zero and
  the height variance grows at most logarithmically with time.

 \end{enumerate}

 This state of affairs naturally leads to two questions. The first is
 to understand what happens when the function $v(\cdot)$ is not
 differentiable, so that the Hessian matrix $H_\rho$ in \eqref{eq:SPDE} is
 ill-defined. Can singularities of $v(\cdot)$ lead to qualitatively
 different dynamic phenomena?  Secondly, Wolf's picture predicts rough
 stationary states, i.e. such that the l.h.s. of \eqref{eq:alpha}
 diverges at larges distances, either logarithmically (AKPZ class) or
 as a power law (KPZ class). On the other hand, it is well known that
 certain \emph{equilibrium}, two-dimensional random interface models
 exhibit smooth phases, i.e., for some choices of the slope the
 l.h.s. of \eqref{eq:alpha} is bounded uniformly in $x,y$. Classical
 example are Solid-on-Solid interfaces and $+/-$ interfaces of the
 three-dimensional Ising model with Dobrushin boundary conditions, for
 slope $\rho=0$ and sufficiently low temperature \cite{
   brandenberger1982decay,dobrushin1973gibbs,bricmont1982surface}. Is
 it possible to have non-trivial AKPZ growth models with a smooth phase
 and, if yes, how does this fit in Wolf's conjectured picture? (See
 also \cite[Sec. 7]{krug1992kinetic} for a related phenomenon in a
 \emph{one-dimensional} growth model and for the discussion of the
 corresponding ``faceting transition'').
 
 In this work, through the analysis of an explicit model, we show that
 the two questions are closely related. The growth process we study is
 based on the so-called shuffling algorithm \cite{EKLP92,Pro03}, that
 was devised to perfectly sample domino tilings of the Aztec
 diamond. Instead of working on the Aztec diamond, the dynamics are
 defined on the (toroidal) periodized lattice
 $(\mathbb Z/2L\mathbb Z)^2$ and eventually on the infinite lattice
 $\mathbb Z^2$, a framework that is more relevant for growth
 processes.  The content of Theorem \ref{th:invinf} is that the Gibbs
 measures \cite{KOS03} for the dimer model on $\mathbb Z^2$, with
 2-periodic weights are translation invariant, stationary,
 non-reversible measures of the dynamics. The reason behind this
 choice of weights is that this is the simplest and best-studied
 \cite{CJ16, DK17} situation that admits a smooth phase: the Gibbs
 measure has $O(1)$ fluctuations if the slope $\rho$ is zero.  (As a
 side remark, what we call here ``rough'' and ``smooth'' phases
 correspond to ``liquid'' and ``gas'' phases in the language of
 \cite{KOS03}).  As far as growth of fluctuations in time, Theorem
 \ref{th:vrho} yields that the l.h.s. of \eqref{eq:beta} grows at most
 logarithmically in $t$ in the rough phase $\rho\ne0$, as is typical
 for models in the AKPZ class, and is $O(1)$ uniformly in time in the
 smooth phase $\rho=0$.  Most of the technical work is devoted to
 actually computing the speed of growth $v(\rho)$. The ``implicit''
 formula for the speed in terms of dimer occupation probabilities is
 very simple, see Eq. \eqref{eq:vrho}. On the other hand, in order to
 analyze its convexity properties and its singularity for $\rho\to0$,
 we need a more workable expression. This is the content of Theorem
 \ref{th:v}. Note that, while the speed vanishes for $\rho=0$, the
 stationary measure is \emph{non-reversible} also in this case.

 Even with an
 explicit expression for $v(\cdot)$ like Eq. \eqref{eq:v} at hand,
 formulas for its second derivatives are too complicated to check
 directly whether the sign of $\det H_\rho$ is negative, as suggested
 by Wolf's picture. To bypass these difficulties, we found an
 equivalent but more complex-analytic type expression for $v(\cdot)$,
 see Theorem \ref{th:harmonic}, that allows to prove $\det H_\rho<0$
 for every $\rho\ne0$.  Finally, for $\rho\to0$ we show that one
 eigenvalue of $H_\rho$ tends to $-\infty$ and the other tends to $0$,
 while their product tends to a negative constant.

 It is natural to ask whether the latter behavior is common to other
 AKPZ growth models, when the slope approaches that of a smooth phase.
 The shuffling algorithm on $\mathbb Z^2$, with weights of period
 larger than $2$, may provide examples where such conjecture could be
 tested, provided that the {face weights} evolve periodically with the
 shuffling algorithm (this is not true in general, see discussion in the next paragraph).  Let us remark
 that if the space  periodicity of the graph is high enough, there may exist
 several smooth stationary phases for different integer values of the
 slope $\rho\in\mathbb Z^2$ \cite{KOS03}. It is presently unclear how
 to extend to this more general case the results of Theorems
 \ref{th:v}--\ref{th:AKPZ}. The main obstacle is that part of our
 approach relies on an explicit computation of the speed of growth, a
 task which seems, at least to us, to increase in complexity as the
 periodicity increases.

To conclude, let us mention
 an intriguing connection between the growth process we study and a
 discrete dynamical system introduced by Goncharov and
 Kenyon~\cite{GK13}. Put simply, one of the results of \cite{GK13} is
 that the shuffling algorithm on the torus, seen as a map on weighted
 periodic graphs, behaves similar to a classical Hamiltonian
 integrable dynamical system. More precisely, some quantities are
 conserved, such as the {spectral curve} of the associated dimer
 model, while (introducing a suitable algebraic structure) others
 evolve quasi-periodically in time, such as the face weights. The
 toroidal square grid with two-periodic weights, that we consider
 here, belongs to a special class of graphs where the evolution of the
 face weights is not just quasi-periodic but actually periodic in
 time.

The rest of the paper is organized as follows: in
Section~\ref{sec:notations} we give the necessary notation for the
paper while in Section~\ref{sec:dynamics}, we describe the shuffling
dynamics on the torus and give our main results.  Proof that the
dynamics are stationary is given in Section~\ref{sec:firstproofs} and
the proof of the formula for the speed (in terms of edge
probabilities) and of the bound on fluctuation growth is given in
Section~\ref{sec:speedfluctuationproofs}. The explicit formula for the
speed of growth for the two-periodic weights is found in
Section~\ref{sec:compv} and in Section~\ref{sec:AKPZsig} it is shown
that the model is in the AKPZ class when $\rho \not = 0$.

\section{Notations}
\label{sec:notations}
We let $\mathbb T_L$ be the square lattice $\mathbb Z^2$ periodized,
both horizontally and vertically, with period $2L$. The graph
$\mathbb T_L$ is bipartite and the $(2L)^2$ vertices are alternately
colored black and white.  Faces of $\mathbb T_L$ with a white top-right
vertex are called ``even'', the others are
called ``odd''.  Denote $\Omega_L$ to be the set of dimer coverings (perfect
matchings) of $\mathbb T_L$.   Let $C_1$
(resp. $C_2$) be a horizontal (resp. vertical) nearest-neighbor, oriented
closed path of length $2L$ on faces, directed to the right (resp. upward).  Given
$\eta\in \Omega_L$, we let for $i=1,2$
\begin{eqnarray}
  \label{eq:Delta1}
  \Delta_i(\eta)=\sum_{e\in C_i} \sigma_e 1_{e\in\eta}
\end{eqnarray}
where the sum is over edges crossed by the path, $\sigma_e$ is $+1$
if the edge $e$ is crossed with the white vertex on the right and $-1$
if it is crossed with the white vertex on the left
and $1_{e\in \eta}$ is the indicator that $e$
is occupied by a dimer.  Correspondingly, given $\Delta=(\Delta_1,\Delta_2)$, let
\begin{eqnarray}
  \label{eq:OmegaL}
  \Omega_L(\Delta)=\{\eta\in\Omega_L: \Delta_i(\eta)=\Delta_i, i=1,2\}.
\end{eqnarray}
The ``slope''
$(\Delta_1(\eta)/L,\Delta_2(\eta)/L)$ belongs \cite{KenLectures} to $ \mathcal P$, the 
closed square with vertices $(-1,0),(0,-1),(1,0),(0,1)$.

Edges of $\mathbb T_L$ are assigned a positive weight taking values
either $1$ or $a>0$ (without loss of generality, we let $0<a\le
1$). First, assign to each even face a label ``$1$'' or ``$a$'' in an
alternate way, as in Fig.~\ref{fig:weightsfaces}.
\begin{figure}
\begin{center}
\includegraphics[height=5cm]{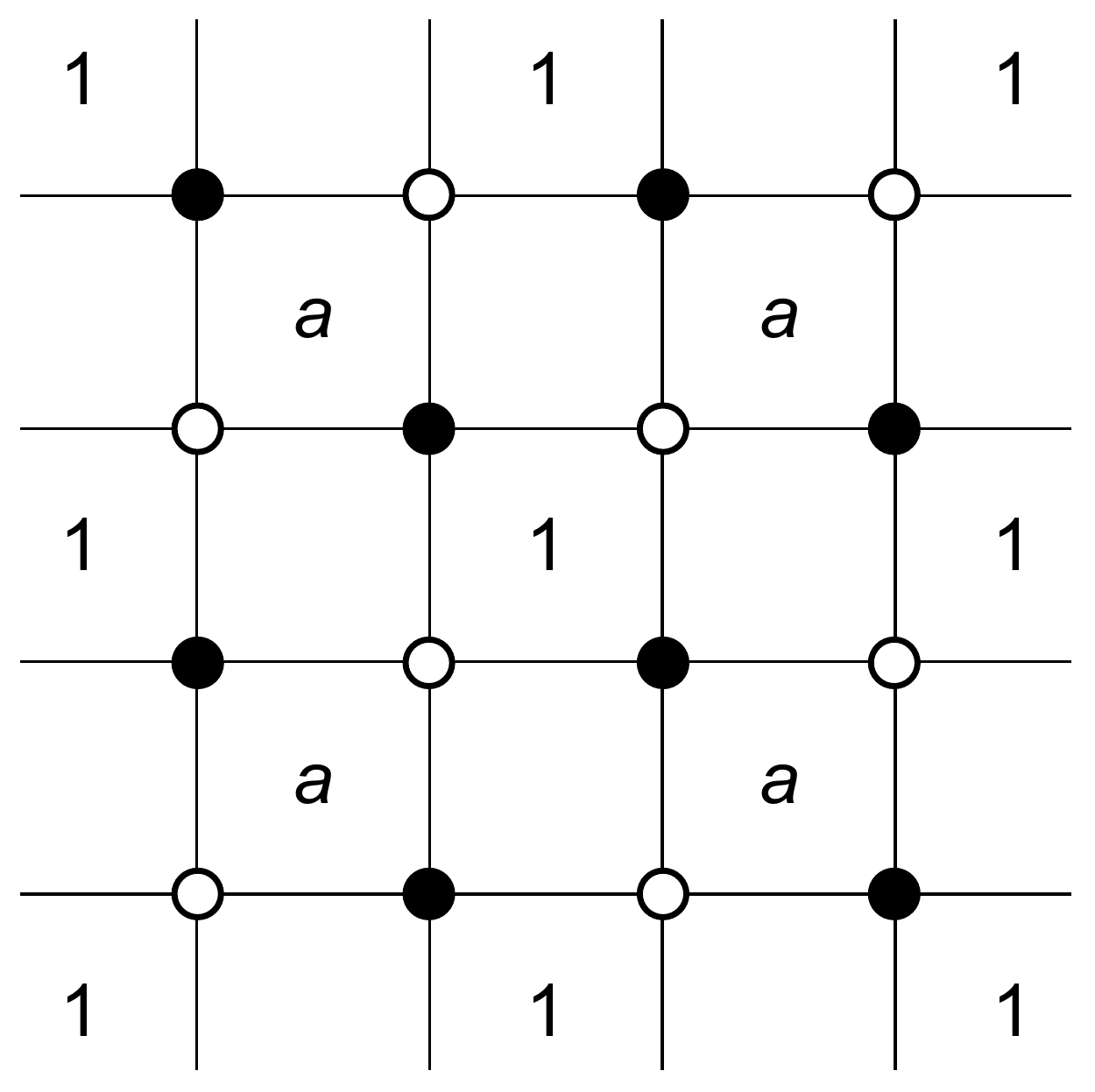}
	\hspace{10mm}
\includegraphics[height=5cm]{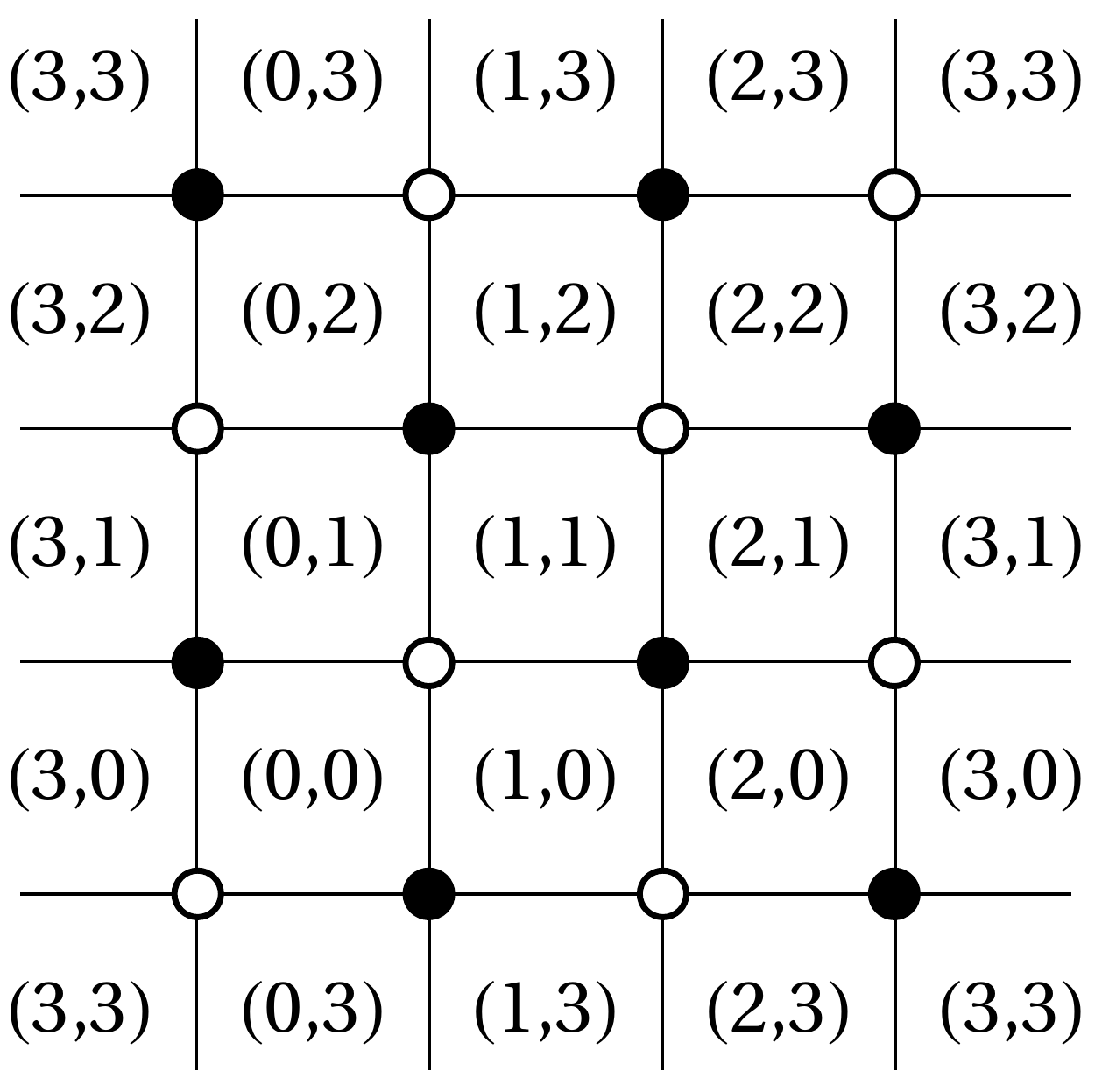}
\caption{The left figure shows the two-periodic weights on $\mathbb{T}_2$: edges incident to the face labeled $a$ take edge weight $a$ while edges incident to the face labeled $1$ take edge weight $1$. The right figure shows the coordinates for the faces of $\mathbb{T}_2$. In both figures, the half edges wrap around the torus.}
\label{fig:weightsfaces}
\end{center}
\end{figure}
Then, we establish that the weight of
an edge is given by the label of the even face it belongs to. Note
that weights are $2$-periodic, i.e. they are left invariant by a horizontal or vertical translation of even  length.

 For $\Delta:=\Delta^{(L)}=(\Delta_1,\Delta_2)$
such that $\Delta/L$ in $\stackrel\circ{\mathcal P}$ (the interior of
$\mathcal P$), let $\pi^{(L)}_{\Delta}$ be the Boltzmann-Gibbs measure on
$\Omega_L(\Delta_1, \Delta_2 )$ with weight proportional to $
a^{N_a(\eta)}$,
with $N_a(\eta)$ the number of dimers on edges of type $a$. It is known \cite{KOS03} that, if
\[
  \lim_{L\to\infty} \frac{\Delta^{(L)}}L=\rho\in\stackrel\circ{\mathcal P},
  \]
  then the measure $\pi^{(L)}_{\Delta^{(L)}}$ has a limit $\pi_{\rho}$
  as $L\to\infty$. The ``infinite-volume Gibbs measure'' $\pi_\rho$ is
  a probability measure on $\Omega_{\mathbb Z^2}$, the set of perfect
  matchings of the infinite lattice $\mathbb Z^2$, invariant and
  ergodic with respect to translations by vectors
  $v\in(2\mathbb Z)^2$, and convergence holds for all local bounded
  functions $f$.  As discussed in the next section, the average ``height
  function'' under $\pi_{\rho}$ is affine with slope $\rho$. The limit
  measure $\pi_{\rho}$ has determinantal correlation functions, with an
  explicit kernel (that depends on $\rho$ and $a$), see Section
  \ref{sec:compv}.  In the same section we will recall also the
  following fact: with the nomenclature of \cite{KOS03}, the Gibbs
  measure $\pi_\rho$ corresponds to a ``liquid phase'' (with power-law
  decaying correlations) if
\begin{eqnarray}
\label{eq:calL}
  \rho\in \mathcal L:=\left\{
  \begin{array}{lll}
\stackrel \circ{\mathcal P}\setminus\{0\} & \text{ if } &a<1\\
\stackrel \circ{\mathcal P}&\text{ if }& a=1    
  \end{array}
\right.
\end{eqnarray}
and to a gaseous phase (with exponentially decaying correlations) if $a<1,\rho=0$.
In the following, we use the terms ``rough'' and ``smooth'' instead of ``liquid'' and ``gaseous''.

\section{Dynamics} \label{sec:dynamics}

\subsection{Spider moves and shuffling algorithm}
\label{sec:general}
The dynamics we study is a version of the \emph{shuffling algorithm
  for the Aztec diamond}~\cite{EKLP92, Ciu98, Pro03} but reworked for the
torus.  Its definition is based on two well-known transformations on
weighted graphs, that we recall here in a general context; see also
\cite{GK13}.  We will come back to the
periodized square lattice with $2$-periodic weights in Section
\ref{sec:actual}.

Here are the two
transformations  of a graph with weighted edges into a new graph
with new edge weights:
\begin{enumerate}
\item (Spider move) Suppose that a graph contains a square face $f$
  with positive edge weights $a,b,c$ and $d$ where the labeling is clockwise
  around the face starting from the topmost horizontal edge.  We
  replace the  face $f$ by a smaller square with edge weights
  $A, B,C $ and $D$ (with the same convention as for $f$),
  and add an edge, with edge weight $1$, between each vertex of the
  smaller square and its original vertex.  Call these added edges
  \emph{legs} and see Fig.\ref{fig:spider} for a
  diagram of this move.  Then, set $A=c/(a c+b d)$, $B=d/(a c+b d)$,
  $C=a/(a c+b d)$ and $D=b/(a c+b d)$. This transformation is called the
  \emph{spider move at face $f$}.
\begin{figure}
	\begin{center}
		\includegraphics[height=4cm]{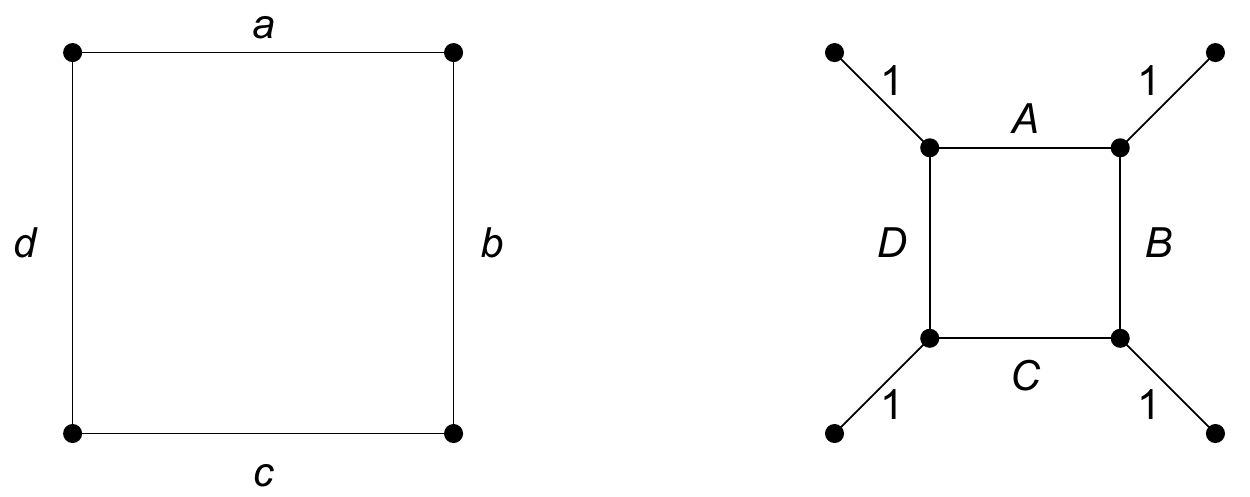}
		\caption{The spider move transformation.}
		\label{fig:spider}
	\end{center}
\end{figure}
  
\item (Edge contraction) Given a weighted graph, define a new
  weighted graph as follows: for any two-valent vertex with the
  incident edges having weight $1$, contract the two incident
  edges. See Fig. \ref{fig:contract}. This procedure is
  called \emph{edge contraction}. 
\end{enumerate}
\begin{figure}
	\begin{center}
		\includegraphics[height=2cm]{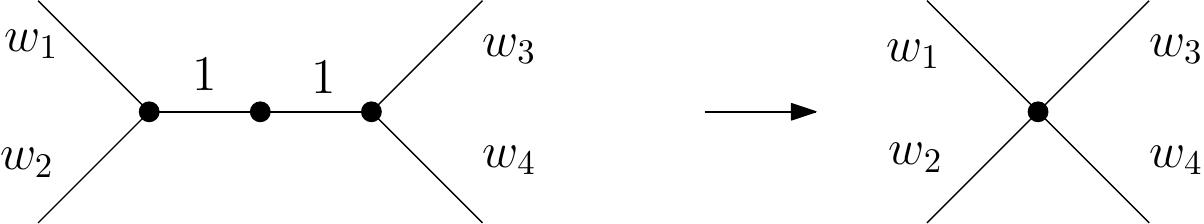}
		\caption{The edge contraction transformation. Weights $w_i$ of the uncontracted edges stay the same.}
		\label{fig:contract}
	\end{center}
\end{figure}

To introduce the dynamics, we first define a sequence of
transformations of the edge weights of the graph $\mathbb T_L$, built
via spider moves and edge contractions. 
We label the faces using Cartesian coordinates $(i,j)$ (both modulo $2L$) so that even faces correspond to $i+j \mod 2=0$ while odd faces correspond to $i+j \mod 2=1$.  For the face $(i,j)$ with $i,j \in \{0,1,\dots, 2L-1\}$, the face to its right has coordinates $(i+1\mod 2L,j)$ while the face above it has coordinates $(i,j+1\mod 2L)$; see Fig.~\ref{fig:weightsfaces}.

For $(i,j) \in \{0,1,\dots,2L-1\}^2$ and $k \geq 0$, let
$w_{i,j;k}^a, w_{i,j;k}^b, w_{i,j;k}^c$ and $w_{i,j;k}^d$ be positive
real numbers
and
$$\Delta_{i,j;k}= w_{i,j;k}^a w_{i,j;k}^c +w_{i,j;k}^b w_{i,j;k}^d.$$
Let
$$\mathtt{w}_k=\{(w_{i,j;k}^a, w_{i,j;k}^b, w_{i,j;k}^c,w_{i,j;k}^d):
(i,j) \in\{0,1,\dots,2L-1\}^2\}$$ where the 4-tuple
$(w_{i,j;k}^a, w_{i,j;k}^b, w_{i,j;k}^c, w_{i,j;k}^d)$ denotes the edge weight around the face $(i,j)$ and at a time $k$, where we use the same convention $a, b, c$ and $d$ as given in
Fig.~\ref{fig:spider}.   Note that $\mathtt{w}_k$ is completely determined by the edge weights around even (resp. odd) faces; we keep track of the weights around both the even and odd faces for convenience.  

The relation between $\mathtt{w}_k$ and
$\mathtt{w}_{k+1}$ is, by definition
\begin{multline}
	\label{eq:wktowk+1}
	(w_{i,j;k+1}^a, w_{i ,j;k+1}^b, w_{i,j;k+1}^c,w_{i,j;k+1}^d)\\:= \left( \frac{w_{i,[j+1]_{2L};k}^a}{\Delta_{i,[j+1]_{2L};k}} ,  \frac{w_{[i+1]_{2L},j;k}^b}{\Delta_{[i+1]_{2L},j;k}}, \frac{w_{i,[j-1]_{2L};k}^c}{\Delta_{i,[j-1]_{2L};k}}, \frac{w_{[i-1]_{2L},j;k}^d}{\Delta_{[i-1]_{2L},j;k}} \right)
\end{multline}
for $k\geq 0$, $i+j\mod 2=(k+1) \mod 2$ and where we have used
$[r]_{2L}=r\mod 2L$ for compactness of notation. Since the weights on
$\mathbb{T}_L$ are fully determined by the edge weights on either the
odd faces or the even faces, then it follows that all weights
$\{\mathtt{w}_k\}_{k\ge1}$ are determined by $\mathtt{w}_0$.  Note
also that the change of weights from $\mathtt{w}_k$ to
$\mathtt{w}_{k+1}$ is equivalent to first applying the spider move to
every face of $\mathbb T_L$ with the same parity (even or odd) as $k$
and then applying the edge contraction transformation to the graph
thus obtained.

For $k\geq 0$ an  integer of even/odd parity $\sigma$, we define a random map
$\Omega_L\ni \eta \mapsto {F}_k(\eta)\in\Omega_L$ through the
following three steps, cf. Fig.~\ref{fig:spiderconfigs} (only the third one is actually random):
\begin{enumerate}
\item [(Deletion step)] All pairs of parallel dimers of $\eta$
  covering two of the four boundary edges of any face of parity
  $\sigma$ are removed.
\item [(Sliding step)]  For every  face of parity $\sigma$ with only one boundary
  edge covered by a dimer of $\eta$,  slide this dimer across that
  face.
\item [(Creation step)] For each $(i,j)\in\{0,1,\dots,2L-1\}^2$, if the
  face of parity $\sigma$ and coordinates $(i,j)$ has no dimers of
  $\eta$ covering any of the four boundary edges, add two parallel
  vertical dimers to the face with probability
  $w^{b}_{i,j;k} w^{d}_{i,j;k}/\Delta_{i,j;k}$ or two parallel
  horizontal dimers with probability
  $w^{a}_{i,j;k} w^{c}_{i,j;k}/\Delta_{i,j;k}$ (the operations are
  performed independently for each $(i,j)$ and $k$).
\end{enumerate}

\begin{figure}
\begin{center}
\includegraphics[height=1.6cm]{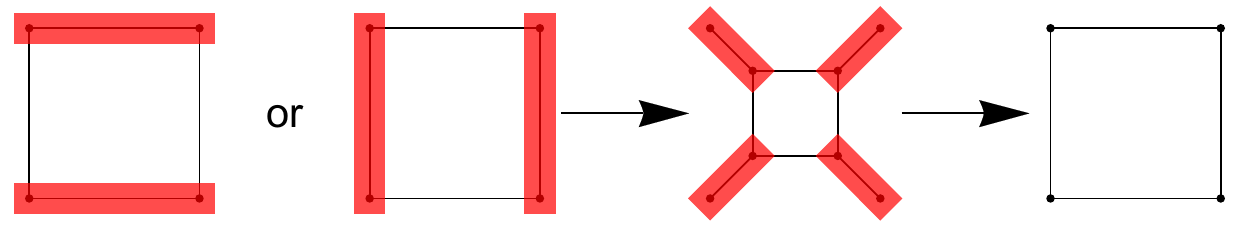}
\vspace{10mm}

	\includegraphics[height=1.6cm]{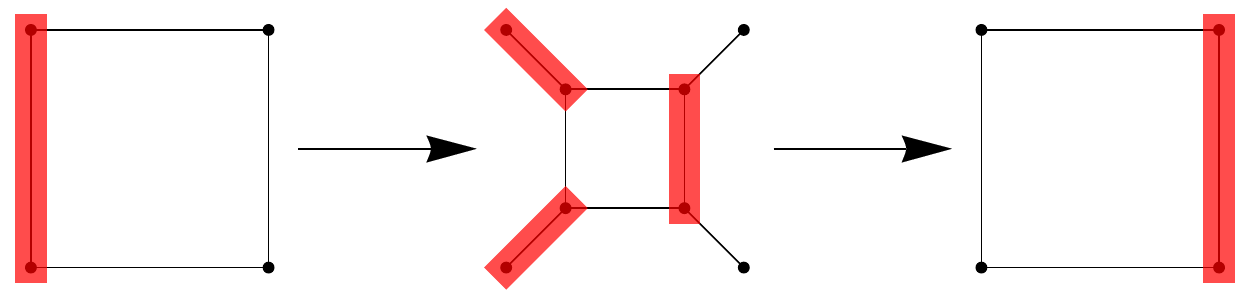}
\vspace{10mm}

\includegraphics[height=1.6cm]{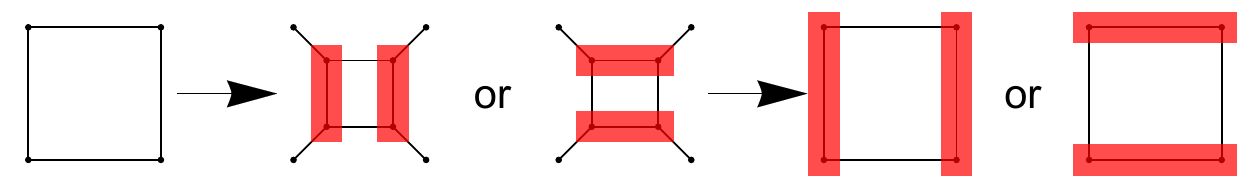}
\caption{The three possible configurations around a face of the same parity as time $k$, and the effect of the transformation $F_k$ (in the second
  line, similar transformations hold when the dimer is on the top,
  bottom or right edge). Red edges denote dimers.
}
\label{fig:spiderconfigs}
\end{center}
\end{figure}

Let ${\pi}_{\Delta;k}^{(L)}$ denote the probability measure (on $\Omega_L(\Delta)$) of the dimer model  on $\mathbb T_L$ with 
weights $\mathtt{w}_k$ and height change $\Delta=(\Delta_1,\Delta_2)$. The following property is 
crucial for the study of the growth model defined in next section:
\begin{prop} \label{prop:general}
Suppose that $\eta \sim {\pi}_{\Delta;k}^{(L)}$.  Then,   ${F}_k(\eta)\sim {\pi}_{-\Delta;k+1}^{(L)}$.
\end{prop}

The above proposition is implicitly known but we could not find an explicit statement in the literature. 

\subsection{The growth process}
\label{sec:actual}

Here we turn back to the case where the graph $\mathbb T_L$ has
$2$-periodic edge weights with values $1$ and $0<a\le 1$ as described
in Section \ref{sec:notations}. Coherently with the conventions of
Section \ref{sec:general}, we call such  weighting $\mathtt{w}_0$.

Let ${\bf e_1}$ (resp. ${\bf e_2}$) be the vector from the face
$(0,0)$ to the face $(1,0)$ (resp. to the face $(0,1)$) and
$\tau_n(\eta)$ be the translation of the dimer configuration $\eta$ by
$n_1{\bf e_1}+n_2{\bf e_2}, n=(n_1,n_2)\in\mathbb Z^2$. Note that
${\bf e_1,e_2}$ are just the usual length-one vectors pointing to the
right and up, respectively.  With some abuse of notation, given a face $f$ or edge $e$ we denote $\tau_n(f),\tau_n(e)$ the face/edge translated by $n_1{\bf e_1}+n_2{\bf e_2}$.
Note that if $n_1,n_2$ are both even, then the vector $n$ connects two faces of the same parity and two edges of the same parity (i.e. both with white edge on top or both with black edge on top).

The random map $T$ that
defines our dynamics is given as follows:
\begin{Definition} \label{def:dynamics}
  Given $\eta\in\Omega_L$, we let
  \begin{eqnarray}
    \label{eq:T}
    T(\eta):=\tau_{(-1,-1)}[F_1\circ F_0(\eta)].
  \end{eqnarray}
\end{Definition}
In other words, we apply first $F_0$, then $F_1$ and then we shift the
configuration one step down and one step to the left on $\mathbb T_L$.
Given an initial condition $\eta(0)\in\Omega_L$, we let
$\eta(k),k\in\mathbb N$ denote the configuration at time $k$, i.e.
the result of the application of $k$ i.i.d. copies of the map $T$ to
$\eta(0)$.

We begin with the basic properties of the dynamics:
\begin{Theorem}
  \label{th:Mchain}
  \label{prop:T}
  The Markov chain defined by the transformation $T$ is ergodic on each sector $\Omega_L(\Delta)$, i.e. it connects any two elements in $\Omega_L(\Delta)$. Moreover, if
   $\eta\sim \pi^{(L)}_{\Delta}$ then $T(\eta)\sim \pi^{(L)}_{\Delta}$.
 \end{Theorem}
 It is easy to deduce that stationarity holds also for the infinite
 dynamics. Indeed, note that the definitions of $F_0,F_1$ and $T$ make
 perfect sense as maps from $\Omega_{\mathbb Z^2}$ to itself (with
 $\Omega_{\mathbb Z^2}$ the set of perfect matchings of the infinite
 lattice $\mathbb Z^2$).  Then:
 \begin{Theorem}
   \label{th:invinf}
   For every local function $g$ and for any slope $\rho$ in $\stackrel \circ{\mathcal P}$ one has
  \begin{eqnarray}
    \label{eq:invarianza}
    \mathbb E_{\pi_{\rho}}g(T(\eta))=\pi_{\rho}(g(\eta)).
  \end{eqnarray}
 \end{Theorem}
 
 A dimer covering of $\mathbb Z^2$ is naturally associated to a height
 function \cite{KenLectures}; we briefly recall this definition and then
 discuss how the height function evolves under the dynamics. This will
 provide an interpretation of the Markov chain as a two-dimensional
 stochastic growth model. Given $\eta\in \Omega_{\mathbb Z^2}$ one
 associates a height function
 $h_\eta:f\in (\mathbb Z^2)^*\mapsto h_\eta(f)$ (defined on the faces
 $f$ of the graph) by fixing it as $h_\eta(f_0)=0$ on a given
 reference face $f_0$, say the even face with coordinates $(0,0)$, and
 by establishing that its gradients are given as\footnote{Note the
   similarity with \eqref{eq:Delta1}. On the torus, the height is a
   multi-valued function that increases by $\Delta_i(\eta)$ along a
   closed path winding in direction $i=1,2$}
\begin{eqnarray}
  \label{eq:hxy}
  h_\eta(f')-h_\eta(f)=\sum_{e\in \mathcal C_{f\to f'}}\sigma_e(1_{e\in\eta}-1/4),
\end{eqnarray}
where $\mathcal C_{f\to f'}$ is any nearest-neighbor path from $f$ to
$f'$ (it is well-known that the r.h.s. of \eqref{eq:hxy} is
path-independent).

 By construction of the finite-volume measures
$\pi^{(L)}_{\Delta}$, one has
\begin{eqnarray}
  \pi_{\rho}[h_\eta(\tau_n(f))-h_\eta(f)]=\frac{n\cdot\rho}2
\end{eqnarray}
if $n\in(2\mathbb Z)^2$.

From Theorem \ref{th:invinf}, we know that if
$\eta(0)\sim \pi_{\rho}$, the law of the height gradients of $\eta(k)$
is time-independent. However, we have not yet determined how the
additive constant $h_{\eta(k)}(f_0)$ changes with time from the
initial value $h_{\eta(0)}(f_0)=0$.  
Assume from now on that the reference face $f_0$ is even.
To motivate our convention given
in Definition \ref{def:newh} below, we start from the following
observation.
\begin{prop}
  \label{prop:dup}
If $f_1,f_2$ are two even faces and $\eta\in\Omega_{\mathbb Z^2}$, then	
	\begin{align}
  \label{eq:dh1}
& h_{\tau_{(-1,0)}[F_0(\eta)]}(f_2)-h_{\tau_{(-1,0)}[F_0(\eta)]}(f_1)=h_{\eta}(\tau_{(1,0)}(f_2))-h_{\eta}(\tau_{(1,0)}(f_1))\\
  \label{eq:dh2}& h_{T (\eta)}(f_2)-h_{T(\eta)}(f_1)=h_{\tau_{(-1,0)}[F_0(\eta)]}(\tau_{(0,1)}(f_2))-h_{\tau_{(-1,0)}[F_0(\eta)]}(\tau_{(0,1)}(f_1)).
\end{align}
\end{prop}
We make therefore the following choice:
\begin{Definition}
  \label{def:newh}
  Given $\eta\in\Omega_{\mathbb Z^2}$, we set
  \begin{eqnarray}
    \label{eq:ddh1}
 &&   h_{\tau_{(-1,0)}[F_0(\eta)]}(f_0):=h_\eta(\tau_{(1,0)}(f_0))\\
    \label{eq:ddh2}
&&    h_{T(\eta)}(f_0):= h_{\tau_{(-1,0)}[F_0(\eta)]}(\tau_{(0,1)}(f_0)).
  \end{eqnarray}
\end{Definition}
This way, the whole height functions of $F_0(\eta)$ and of $T(\eta)$
are completely defined, including the overall additive constant, in terms of that of $\eta$.  
\begin{Remark}
  \label{rem:newh}
  From Proposition \ref{prop:dup}, it follows immediately that
  \eqref{eq:ddh1}-\eqref{eq:ddh2} hold for any even face $f$ if they
  hold for $f_0$; in other words, Definition \ref{def:newh} is
  actually \emph{independent} of the choice of the even face $f_0$.
  One can check that \eqref{eq:ddh1}-\eqref{eq:ddh2} \emph{ do not}, in
  general, hold for odd faces. If they did, the height function
  of $\eta(k)$ would be a deterministic function of the height
  function of $\eta(k-1)$, which is not the case.
\end{Remark}

\subsection{Speed of growth and fluctuations}

\begin{Theorem}[Speed and fluctuations]
  \label{th:vrho}
  Let $f$ be any face of $\mathbb Z^2$, $k\in\mathbb N$ and let for lightness of notation $h_k(\cdot):=h_{\eta(k)}(\cdot)$. Then, for $\rho\in\stackrel\circ{\mathcal P}$,
  \begin{eqnarray}
    \label{eq:vrho}
    v(\rho):=\frac1k\mathbb E_{\pi_{\rho}}(h_{k}(f)-h_0(f))=c_1(\rho)-c_2(\rho)
\end{eqnarray}
with
  \begin{eqnarray}
    \label{eq:c1c2}
    c_1(\rho)=\pi_\rho(e_1\in\eta),\qquad c_2(\rho)=\pi_\rho(e_2\in\eta).
  \end{eqnarray}
where $e_1$ is a horizontal edge with an ``$a$'' face above it and $e_2$ is a vertical edge with an ``$a$'' face to its left.
  One has for any face $f$
  \begin{eqnarray}
    \label{eq:flutt1}
\lim_{u\to\infty}   \limsup_{k\to\infty} \mathbb P_{\pi_{\rho}}(|h_{k}(f)-h_{0}(f)-v(\rho)k|\ge u \sigma(k))=0
  \end{eqnarray}
  where
  \begin{eqnarray}
    \label{eq:N}
    \sigma(k)= \left\{\begin{array}{lll}
      {\sqrt{\log (k+1)}}&\text{if} &
                                                                          \rho\in\mathcal L\\
       1&\text{if} &
                                                                          \rho=0, a<1.
    \end{array}
    \right.
  \end{eqnarray}
  
\end{Theorem}

\begin{Remark}
  \label{rem:vsimm}
  As discussed at the beginning of Section \ref{sec:compv}, the function $v(\cdot)$ satisfies a number of simple reflection symmetries, namely:
  \begin{eqnarray}
    \label{eq:simmv}
    v(\rho_1,\rho_2)&=&-v(\rho_2,\rho_1),\\
    \label{eq:simmv1}
     v(\rho_1,\rho_2)&=&\rho_1-\rho_2-v(-\rho_2,-\rho_1) 
  \end{eqnarray}
  and
  \begin{eqnarray}
    \label{eq:simmv2}
 v(\rho_1,\rho_2)=\rho_2+v(\rho_1,-\rho_2), \quad v(\rho_1,\rho_2)=\rho_1+v(-\rho_1,\rho_2).
  \end{eqnarray}
  As a consequence, $v(\cdot)$ is fully determined by its restriction to the subset
  \[
\mathcal P\cap \{\rho_1\ge0, 0\le \rho_2\le \rho_1\}.
    \]
\end{Remark}

The function $v(\cdot)$ can be explicitly computed (this computation is rather involved and takes the whole Section \ref{sec:compv}):
\begin{Theorem}[Formula for the speed]
  \label{th:v}
  Let $\rho_+=\pi(\rho_1+\rho_2)\in [-\pi,\pi],\rho_-=\pi(\rho_1-\rho_2)\in[-\pi,\pi]$  and $c=a/(1+a^2)\in(0,1/2]$.
 We have
  \begin{eqnarray}
    \label{eq:v}
    v(\rho)=
    \frac{\rho_-}{2\pi}-\frac{\pi-\arccos[\Sigma]}{2\pi}\rm{sign}(\rho_-\rho_+)
  \end{eqnarray}
  where
  \begin{equation}
    \label{eq:Sigma}
    \Sigma=-\frac{\cos(\rho_-)+\cos(\rho_+)+\sqrt{
    (\cos(\rho_-)-\cos(\rho_+))^2+4c^2\sin^2(\rho_-)\sin^2(\rho_+)}}2
  \end{equation}
  and $\arccos(\Sigma)\in[0,\pi]$. 
 \end{Theorem}
The speed of growth is differentiable, except if the weights are genuinely $2$-periodic (i.e. $a<1$) and the slope corresponds to that of the smooth phase:
\begin{Theorem}[Asymptotic behavior of the speed close to the smooth phase]
  \label{th:a}
  The function $v(\cdot)$ is $C^2$ on $\mathcal L$ (recall definition
  \eqref{eq:calL}).  For $a< 1 $ it is not differentiable at $\rho=0$
  and one has the following asymptotic expansion for $\rho\sim 0$. Assume
  $\rho_+>0,\rho_->0 $ 
  and let $r:=\rho_-/\rho_+\in (0,1)$ (the other cases can be obtained by symmetry, see Remark \ref{rem:vsimm}). Then, as $\rho_+\to0$ (under the assumption $a<1$)
  \begin{eqnarray}
    \label{eq:asympt}
	  v(\rho)=\frac{r \rho_+}2 - \frac{\sqrt{f_1(r)}}{2 \pi}  \left[\sqrt 2
\rho_++ \left(\frac{f_1(r)}{6 \sqrt2} + \frac{f_2(r)}{\sqrt2 f_1(r)} \right)\rho_+^3\right]+O(\rho_+^5)
  \end{eqnarray}
  where $f_1(r)>0$.

When $a<1$, one eigenvalue of $H_\rho$
      diverges  as $\rho\to0$, the other tends to zero and $\det(H_\rho)$ tends to a
      finite strictly negative limit.
\end{Theorem}
The explicit expressions of $f_1,f_2$ can be found in Section \ref{sec:propv}. From~\eqref{eq:asympt}, one sees that 
$$ 
\partial_{\rho_+} v(\rho) = -\frac{2}{4 \pi \sqrt{f_1(r)}} \left( 2 f_1(r) -r f_1'(r) \right)+O(\rho_+)
$$
as $\rho_+ \to 0$, which is non-trivial and nowhere vanishing for $c<1/2$. Hence $\partial_{\rho_+} v(\rho)$ depends on $r$ as $\rho \to 0$ and so $\nabla v(\rho)$ is not continuous at $\rho=0$.

As discussed in the introduction, of particular relevance for the
large-scale fluctuation properties of the growth model is the sign of
the determinant of the Hessian matrix $H_\rho$, a negative  or a vanishing determinant
indicating that the model belongs to the AKPZ universality class. 
 For $a<1$, Theorem \ref{th:a} gives that the sign is negative in a neighborhood of $\rho=0$.
  For $a=1$, the formulas for the derivatives of $v(\rho)$ are relatively simple and an explicit computation of the
  determinant shows that
  \begin{eqnarray}
    \label{eq:dethessa1}
\det(H_\rho)<0    
  \end{eqnarray}
 for  every $\rho\in \mathcal L$  (cf. Eq. \eqref{eq:detha1} below).  
 For $a<1$, on the other hand, we see no direct way of proving
 \eqref{eq:dethessa1}, starting from the explicit expression
 \eqref{eq:v} (we did try to simplify the resulting expressions using Mathematica).

 To circumvent this problem, we first give an equivalent,
 complex-analytic characterization of the speed of growth.
Assume that $0<\rho_-,\rho_+<\pi$ and define
\begin{eqnarray}
  \label{eq:XY}
X=\frac{\rho_-}2+\frac\pi2\in(\pi/2,\pi), \qquad   Y=-\frac{\rho_+}2+\frac32\pi\in(\pi,3\pi/2).
\end{eqnarray}

\begin{Theorem}
  \label{th:harmonic}
  For $z\in\mathbb C$, let $G(z)=z-\sqrt{z^2+2c}$ with the branch cut as specified in~\eqref{eq:sqrt}. The mapping from
  $\mathcal Q^+:=\{z\in\mathbb C:\Re(z)>0,\Im(z)>0\}$ to
  $(\pi/2,\pi)\times(\pi,3\pi/2)$ defined by
  $z\in\mathcal Q^+\mapsto (X(z),Y(z))=(\arg(G(z)),\arg G(1/z))$ is a
  diffeomorphism. Here, we adopt the convention that the argument $\arg(\cdot)$ ranges from
  $-\pi/2$ to $+3\pi/2$.

For $0<\rho_-,\rho_+<\pi$ the function $v(\cdot)$ of Theorem \ref{th:v} is equivalently given as follows:
\begin{eqnarray}
  \label{eq:v2}
  v(\rho)=\frac X\pi-1+\frac1\pi\arg z(X,Y),
\end{eqnarray}
with $X=X(\rho),Y=Y(\rho)$ as in \eqref{eq:XY}. 
The other three cases ($-\pi<\rho_-<0<\rho_+<\pi$, $-\pi<\rho_+<0<\rho_-<\pi$ and $-\pi<\rho_-,\rho_+<0$) can be obtained by symmetry.
\end{Theorem}

From the characterization of $v(\cdot)$ provided by Theorem \ref{th:harmonic}, we can prove that the growth model does indeed belong to the AKPZ universality class for $\rho\in\mathcal L$, coherently with the logarithmic upper bound on growth of fluctuations provided by Eq. \eqref{eq:flutt1}:
  \begin{Theorem}[AKPZ signature of the speed of growth]
\label{th:AKPZ}
    For every $a\le 1$ and $\rho\in\mathcal L$, one has
    \begin{eqnarray}
      \label{eq:dh}
      \det(H_\rho)< 0.
    \end{eqnarray}
  \end{Theorem}

\section{
Stationarity of Gibbs measures}
 \label{sec:firstproofs}
In this section, we prove Proposition \ref{prop:general} as well as 
 Theorems \ref{th:Mchain} and \ref{th:invinf}.

\begin{proof}[Proof of Proposition \ref{prop:general}]
  We consider the case $k$ even, as the odd case follows from the same
  argument but interchanging even and odd below. Notice that the whole dimer configuration $\eta$ is
  determined by dimers covering edges on the boundary of the even
  faces and that there is a height change crossing an even face
  horizontally (resp. vertically) if and only if there is exactly one
  vertical (resp. horizontal) dimer covering an edge on the boundary
  of that even face. It then follows immediately that the only step in
  the definition of ${F}_k$ which has an effect on the height function
  is the sliding step: since after sliding the single dimer on the
  boundary of the even face moves to an edge of opposite parity, the
  height change of $F_k(\eta)$ is the opposite as that of $\eta$.

  Take the torus $\mathbb T_L$ with weights $\mathtt{w}_k$ and apply
  the spider move to all the even faces, followed by edge contraction
  of all the resulting two-valent vertices; see Fig.~\ref{fig:squaremovedynamics} for a schematic of the transformation for the underlying graph. 
\begin{figure}
\begin{center}
\includegraphics[height=6cm]{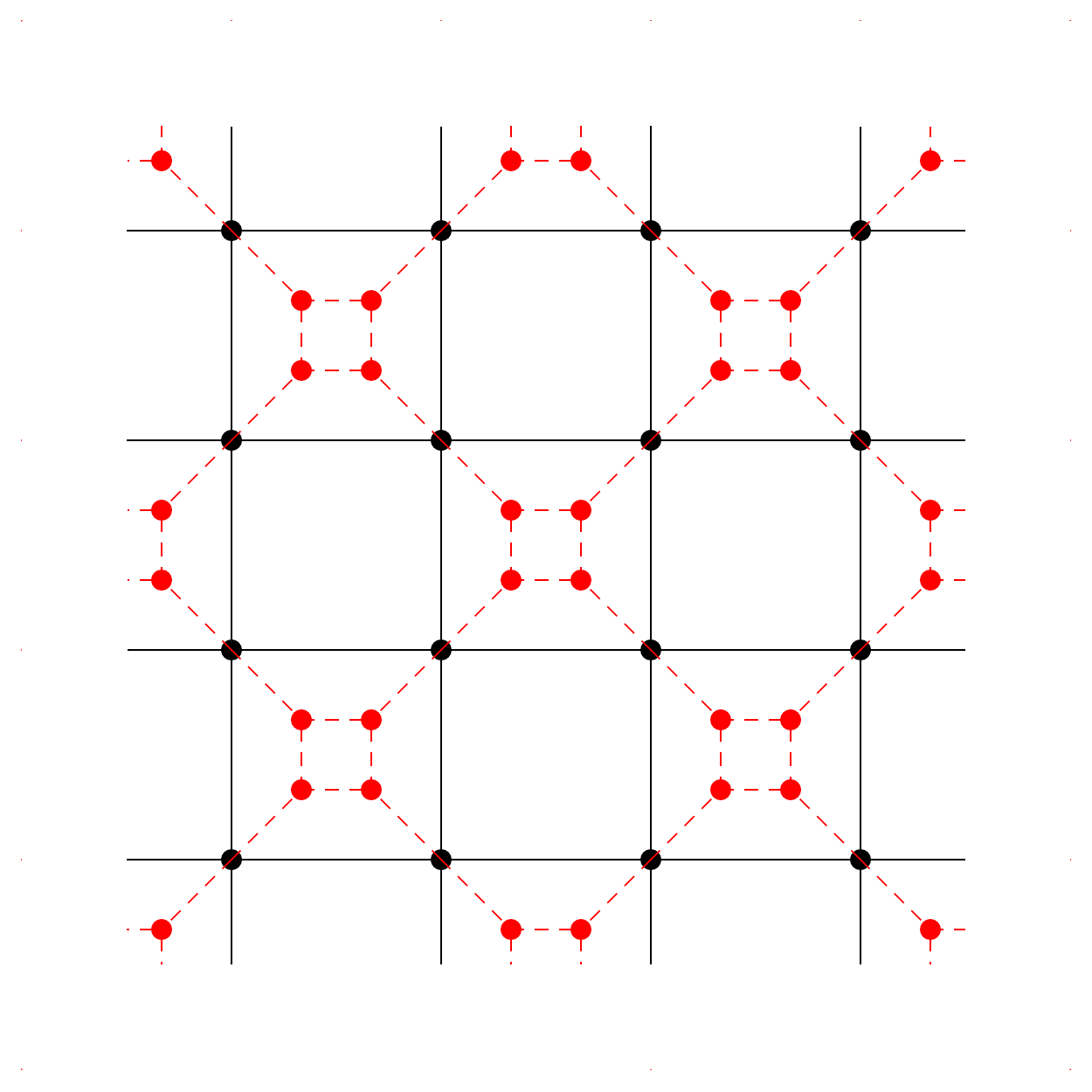}
\includegraphics[height=6cm]{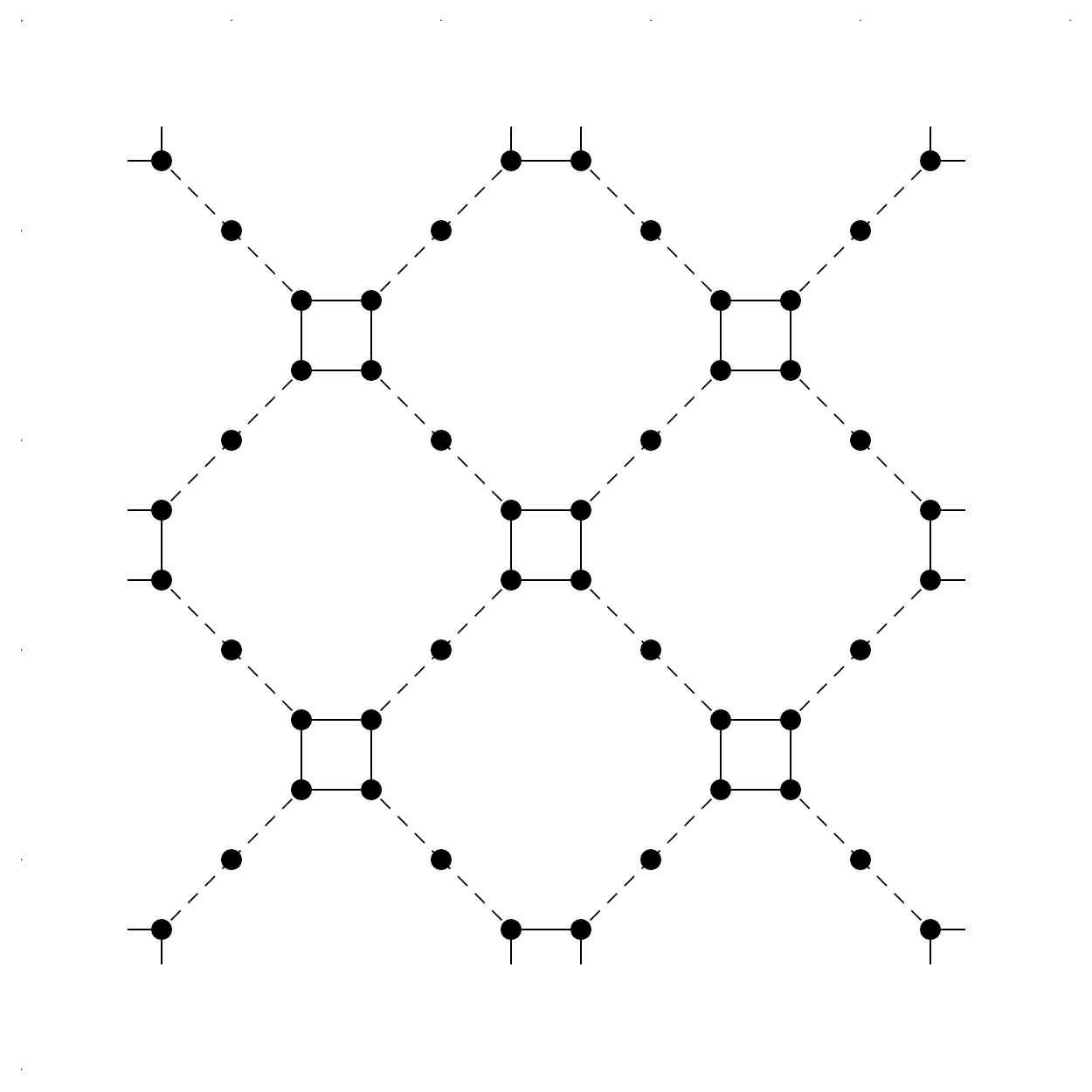}
\caption{The left figure applies the square move, drawn in dashed red, to the even faces of $\mathbb{T}_2$. All original edges (but not the vertices) of the graph are removed from by this procedure and the right figure shows the end result. The dashed edges in the right figure are those that are incident to two-valent vertices and so will be contracted under the edge contraction procedure. In both figures, the half edges represent edges which wrap around the torus.  
  }
\label{fig:squaremovedynamics}
\end{center}
\end{figure}
By construction, one
  obtains the graph $\mathbb T_L$ with weights $\mathtt{w}_{k+1}$.
	Explicitly, the weights around the even face with coordinates $(i,j)$ are given by  
	\begin{equation}\label{newweights}
\left( \frac{w_{i,j;k}^c}{\Delta_{i,j;k}} ,  \frac{w_{i,j;k}^d}{\Delta_{i,j;k}}, \frac{w_{i,j;k}^a}{\Delta_{i,j;k}}, \frac{w_{i,j;k}^b}{\Delta_{i,j;k}} \right),
	\end{equation}
	using the same labeling conventions as above.

	To see that ${F}_k(\eta)\sim {\pi}_{-\Delta;k+1}^{(L)}$, we proceed as
follows. Call $P_k$  the transition matrix from $\Omega_L(\Delta)$ to
$\Omega_L(-\Delta)$ of the random map $F_k$ and let $w^k_e$ be the
weight of the edge $e$ as determined by $\mathtt{w}_k$. If we prove that,
for every $\eta'\in\Omega_{L}(-\Delta)$,
\begin{eqnarray}
  \label{eq:ifwe}
  \sum_{\eta\in\Omega_L(\Delta)}\left[\prod_{e\in\eta}w^k_e\right] P_k(\eta,\eta')=N \prod_{e\in\eta'}w^{k+1}_e
\end{eqnarray}
for some constant $N$ independent of $\eta'$, then the claim of the Proposition follows and actually $N$ equals the ratio of partition functions of $\pi_{\Delta;k}$ and  $\pi_{-\Delta;k+1}$.

Call $A(\eta')$ the set of even faces such that $\eta'$ has no dimer
along the boundary, $D(\eta')$ the set of even faces such that two
parallel boundary edges are covered by dimers of $\eta'$ and
$R(\eta')$ the set of even faces that are neither in $A(\eta')$ nor in
$D(\eta')$.  Note first of all that the set $S(\eta'):=\{\eta$ such
that $P_k(\eta,\eta')\ne0\}$ is the set of configurations such that:
(a) for every face in $A(\eta')$, $\eta$ has two parallel dimers
(either vertical or horizontal) along the boundary of the face; (b)
for every face in $D(\eta')$, $\eta$ has no dimer along the face and
(c) for every even face in $R(\eta')$, $\eta$ has a single dimer along
its boundary, on the edge that is opposed to the one on which $\eta'$
has a dimer.  Secondly, observe that by the definition of the creation
step, for all $\eta\in S(\eta')$, $P_k(\eta,\eta')$ equals the product
over $(i,j)\in D(\eta')$ of
\[
  \frac{w^{r_{i,j}(\eta')}_{i,j;k}w^{r'_{i,j}(\eta')}_{i,j;k}}{\Delta_{i,j;k}}=\Delta_{i,j;k}\frac{w^{r_{i,j}(\eta')}_{i,j;k}}{\Delta_{i,j;k}}\frac{w^{r'_{i,j}(\eta')}_{i,j;k}}{\Delta_{i,j;k}}
\]
where $(r_{i,j}(\eta'),r'_{i,j}(\eta'))$ equals $(a,c)$ or $(b,d)$ according to whether $\eta'$ has two parallel horizontal or two parallel vertical dimers at even face $(i,j)$.
Next, note that for  all $\eta\in S(\eta')$, one has
\begin{eqnarray}
  \prod_{e\in\eta}w^k_e=\prod_{(i,j)\in A(\eta')}\left(w^{r_{i,j}(\eta)}_{i,j;k}w^{r'_{i,j}(\eta)}_{i,j;k}\right)\prod_{(i,j)\in R(\eta')}\left(\frac{w_{i,j;k}^{p_{i,j}(\eta)}}{\Delta_{i,j;k}}\Delta_{i,j;k}
  \right)
\end{eqnarray}
where $(r_{i,j}(\eta),r'_{i,j}(\eta))$ are as above, while
$p_{i,j}(\eta)$ is $a,b,c,d$ according to the position of the unique edge of $\eta$ along the boundary of the even face $(i,j)$. 
Multiplying by $P_k(\eta,\eta')$ and then summing over $\eta\in S(\eta')$ (i.e. over the two possible orientations of dimers of $\eta$ at faces in $A(\eta')$), we see that the l.h.s. of \eqref{eq:ifwe} equals
\begin{eqnarray}
  \label{eq:ifwe2}
  \prod_{e\in\eta'}w^{k+1}_e\,\prod \Delta_{i,j;k}
\end{eqnarray}
where the product in the r.h.s. runs over all even faces of $\mathbb T_L$. We used the fact that, after moving a dimer in the sliding step and changing weights from $\mathtt{w}_k$ to $\mathtt{w}_{k+1}$, the weight assigned to the dimer is  divided by $\Delta_{i,j;k}$.
In conclusion, \eqref{eq:ifwe} is proven with $N=\prod \Delta_{i,j;k}$.
\end{proof}

\begin{proof}[Proof of Theorem \ref{th:Mchain}] 
	For the two-periodic weighting,  $\mathtt {w}_0$ is determined by 
$$
	(w^a_{i,j;0},w^b_{i,j;0},w^c_{i,j;0},w^d_{i,j;0}) = \left\{ \begin{array}{ll}
		a(1,1,1,1)  & \mbox{if } (i,j)\!\!\!\mod 2=(0,0)\\
		(1,1,1,1) & \mbox{if } (i,j)\!\!\!\mod 2 =(1,1).
	\end{array} \right.
$$
It follows from Proposition \ref{prop:general} that
$[F_1\circ F_0](\eta)\sim \pi^{(L)}_{\Delta,2}$.  We have just to
prove that if we translate by $-{\bf e_1}-{\bf e_2}$ the weights
$\mathtt{w}_2$ we obtain the original two-periodic weights
$\mathtt {w}_0$, up possibly to an overall positive prefactor that
multiplies all weights and is inessential in the definition of the measure.

A simple computation based on \eqref{eq:wktowk+1} gives that
	\[
		\begin{split}
(w^a_{i,j;1},w^b_{i,j;1},w^c_{i,j,1},w^d_{i,j;1})
			&= \left\{ \begin{array}{ll}
	\frac{1}{2a}(1,1,1,1)  & \mbox{if } (i,j)\!\!\!\mod 2=(0,0) \\
				\frac{1}{2}(1,1,1,1) & \mbox{if } (i,j)\!\!\!\mod 2 =(1,1). 
	\end{array} \right. \\
			&=	
	\left\{ \begin{array}{ll}
	\frac{1}{2}(1,1/a,1,1/a)  & \mbox{if } (i,j)\!\!\!\mod 2=(1,0)\\
		\frac{1}{2}(1/a,1,1/a,1) & \mbox{if } (i,j)\!\!\!\mod 2 =(0,1).
	\end{array} \right.
		\end{split}
\]

and
$$
(w^a_{i,j;2},w^b_{i,j;2},w^c_{i,j;2},w^d_{i,j;2})
= 2c\,\left\{ \begin{array}{ll}
	(1,1,1,1)  & \mbox{if } (i,j)\!\!\!\mod 2=(0,0) \\
		a(1,1,1,1) & \mbox{if } (i,j)\!\!\!\mod 2 =(1,1)
\end{array} \right.
	$$
        with $c=a/(1+a^2)$ as usual.  Translating by
        $-{\bf e_1}-{\bf e_2}$ the weights $\mathtt{w}_2$, one is back to the original weights
        $\mathtt {w}_0$, up to the global prefactor $2c$, as wished.
\begin{Remark}
\label{rem:ag}
  Note also that, modulo an overall multiplicative constant, the
  weights $\mathtt{w}_1$ correspond to interchanging the positions of
  ``$a$'' and ``$1$'' faces in the original $2$-periodic weighting
  $\mathtt{w}_0$.
\end{Remark}

As far as ergodicity is concerned, assume that there are $n>1$
ergodicity classes within $\Omega_L(\Delta)$, i.e. subsets
$\mathcal C_1,\dots \mathcal C_n$ of $\Omega_L(\Delta)$ that are
invariant for the dynamics.  On the other hand, it is known that any
two configurations in $\Omega_L(\Delta)$ can be connected by a chain
of elementary rotations (i.e. the flip of two parallel adjacent dimers
from vertical to horizontal or vice-versa) . It follows that there
exist two configurations, $\eta\in \mathcal C_i,\eta'\in\mathcal C_j$
with $i\ne j$ that differ by a single elementary rotation at a face
$f$. Assume first that $f$ is an even face at single faces. Then,
apply the transformation $T$ to both $\eta$ and $\eta'$. Since in the
``delete step'' of $F_0$ the discrepancy at the face $f$ disappears,
the two updates can be coupled so that $T(\eta)=T(\eta')$. This
contradicts the assumption that $\eta,\eta'$ belong to two different
ergodicity classes.  Assume that $f$ is an odd face instead and call
$\Omega_\eta\subset \mathcal C_i,\Omega_{\eta'}\subset \mathcal C_j$
the subset of configurations from which one can reach $\eta,\eta'$ via
a single application of $T$. If $i\ne j$, then
$\Omega_\eta\cap\Omega_{\eta'}=\emptyset$.  Let
$\sigma\in \Omega_\eta$ and
$\eta=T(\sigma)=\tau_{(-1,-1)}[F_1\circ F_0](\sigma)$; note that the
two parallel dimers that $\eta$ has around the face $f$ have been
added in the ``addition'' step of $F_1$. On the other hand, the two
parallel dimers at $f$ could have been given (in the same addition
step) the opposite orientation, with positive probability. In this
case, the resulting configuration would be $\eta'$ instead, so that
$\sigma\in\Omega_{\eta'}$. This contradicts the assumption that
$\Omega_\eta\cap\Omega_{\eta'}=\emptyset$ and also contradicts that
there is more than one ergodicity class within $\Omega_L(\Delta)$.
\end{proof}

\begin{proof}[Proof of Theorem \ref{th:invinf}] We give only a sketchy
  proof since the argument is standard.  Call $U_f$ the support of
  $f$. The configuration $T(\eta)$ restricted to $U_f$ depends only on
  the restriction of $\eta$ to $\hat U_f$, the set of edges within
  distance $d$ from $U_f$, where $d$ is an integer independent of $f$.
  Then, since the measure $\pi^{(L)}_{\lfloor L \rho\rfloor}$
  converges locally to $\pi_\rho$ \cite{KOS03}, we can couple two random
  configurations sampled from the two measures in such a way that with
  high probability (as $L\to\infty$) they coincide in $\hat U_f$
  . Then, \eqref{eq:invarianza} follows immediately from the
  stationarity of $\pi^{(L)}_{\Delta}$ for the finite-volume dynamics.
\end{proof}

\begin{Remark} In \cite{toninelli20172+}, one of us proved that the Gibbs
  measures $\pi_\rho$ of the dimer model on the infinite hexagonal
  graph are stationary for an irreversible Markov chain where
  particles can perform unbounded jumps. In that case, deducing
  stationarity for the infinite system from stationarity on the torus
  required non-trivial arguments since the dynamics are non local
  (i.e. in that case it is not true that the configuration at time $1$
  in a domain $U$ depends only on the initial condition in a
  configuration-independent neighborhood of $U$).
\end{Remark}

\section{Proof of Proposition \ref{prop:dup} and Theorem \ref{th:vrho}}
\label{sec:speedfluctuationproofs}

\begin{proof}[Proof of Proposition \ref{prop:dup}]
  We will prove only \eqref{eq:dh1} as the proof of \eqref{eq:dh2} is  
 essentially identical.  It is enough to prove \eqref{eq:dh1} for
  faces $f_1,f_2$ that are separated by a single odd face.  Suppose
  for instance that $f_2$ is to the right of $f_1$.  For lightness
  of notation $\eta':=\tau_{(-1,0)}[F_0(\eta)]$. Then, with the notations of Fig. \ref{fig:dyn1}, the l.h.s. of
  \eqref{eq:dh1} equals
    \[
1_{e_2\in \eta'}-1_{e_1\in\eta'}=1_{e_2\in\eta',e_1\not\in\eta'}-1_{e_1\in\eta',e_2\not\in\eta'}.
    \]
\begin{figure}
	\begin{center}
		\includegraphics[height=4cm]{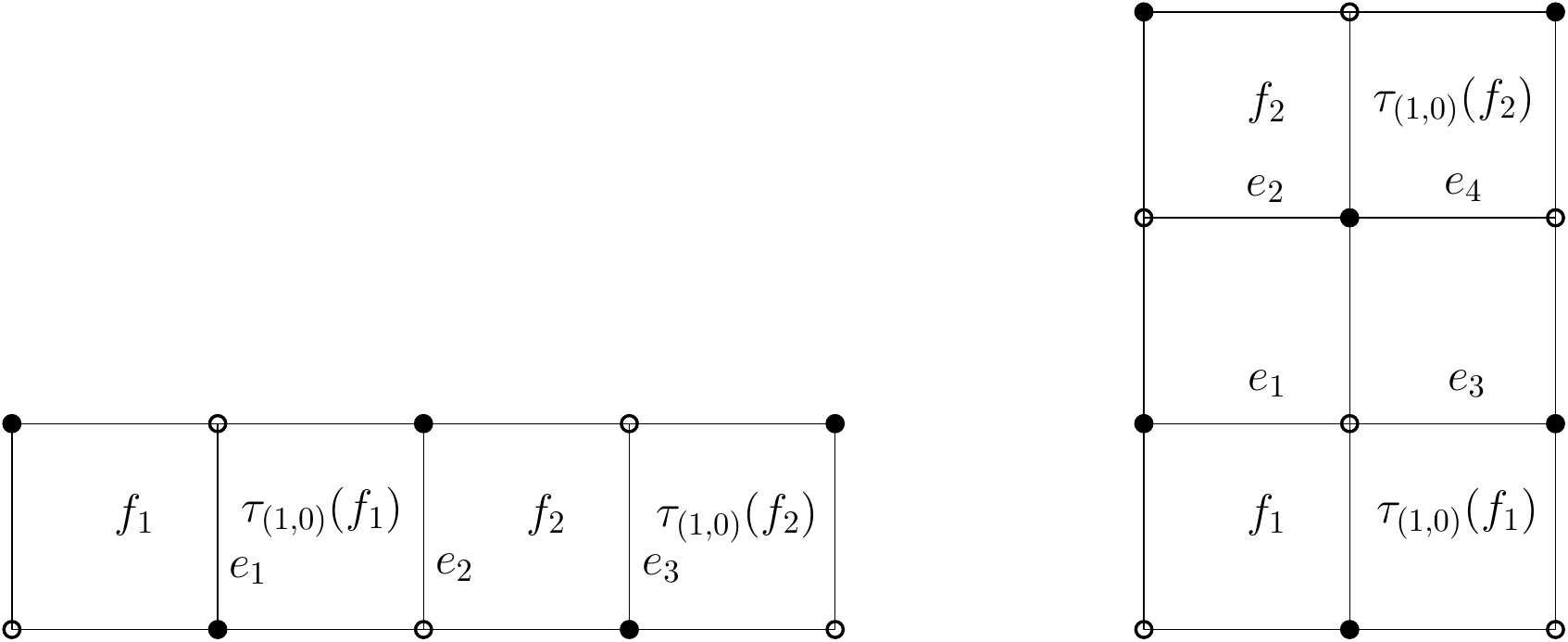}
		\caption{The even faces $f_1,f_2$ together with their
                  translates $\tau_{(1,0)}(f_1),\tau_{(1,0)}(f_2)$
                  when $f_2$ is either to the right or above $f_1$, and
                  the edges $e_1,\dots,e_4$ of the proof.}
		\label{fig:dyn1}
	\end{center}
\end{figure}
From the definition of the ``sliding step'' of $F_0$, one easily checks that the event
$\{e_2\in\eta',e_1\not\in\eta'\}$ is equivalent to the event that
$\{e_2\in\eta,e_3\not \in\eta\}$. Similarly, the event
$\{e_1\in\eta',e_2\not\in\eta'\}$ is equivalent to
$\{e_2\not\in\eta,e_3 \in\eta\}$.  Therefore, the l.h.s. of
\eqref{eq:dh1} equals
    \begin{eqnarray}
      1_{e_2\in\eta,e_3\not \in\eta}-1_{e_2\not\in\eta,e_3 \in\eta}=1_{e_2\in\eta}-1_{e_3\in\eta}
    \end{eqnarray}
    which equals the r.h.s. of \eqref{eq:dh1}.

    If instead $f_2$ is above $f_1$ then we see that the l.h.s. of \eqref{eq:dh1} equals
    \begin{eqnarray}
      1_{e_1\in\eta'}-1_{e_2\in \eta'}=1_{e_1\in\eta', e_2\not\in\eta'}-1_{e_2\in\eta',e_1\not\in\eta'}. 
    \end{eqnarray}
    From the definition of $F_0$ this is easily seen to be equal to
    \begin{eqnarray}
     1_{e_4\in\eta,e_3\not\in\eta}-1_{e_3\in\eta,e_4\not\in\eta}=1_{e_4\in\eta}-1_{e_3\in\eta}
    \end{eqnarray}
    which equals the r.h.s. of \eqref{eq:dh1}.
    \end{proof}

    \subsection{Proof of Theorem \ref{th:vrho}}
    \begin{proof}[Proof of \eqref{eq:vrho}]
      Since the process is stationary, it is clear that the second
      expression in \eqref{eq:vrho} is independent of $k $ and of
      $f$. We will therefore take $k=1$ and choose $f$ to be an even
      face of type $a$. From Definition \ref{def:newh} and Remark
      \ref{rem:newh} we see that the height increase in a step at $f$
      is
\begin{multline}
  \label{eq:perplesso}
h_{T(\eta)}(f)-h_{\eta}(f)\\=h_{\tau_{(-1,0)}[F_0(\eta)]}(\tau_{(0,1)}(f))-h_{\tau_{(-1,0)}[F_0(\eta)]}(f)+h_{\tau_{(-1,0)}[F_0(\eta)]}(f)-h_{\eta}(f)
\\=[h_{\tau_{(-1,0)}[F_0(\eta)]}(\tau_{(0,1)}(f))-h_{\tau_{(-1,0)}[F_0(\eta)]}(f)]+[h_\eta(\tau_{(1,0)}f)-h_\eta(f)].
\end{multline}
Now we take expectation over $\eta\sim\pi_\rho$ and obtain
\begin{eqnarray}
  \pi_\rho[h_\eta(\tau_{(1,0)}f)-h_\eta(f)]=1/4-c_2(\rho).
\end{eqnarray}
 On the other hand, 
\begin{eqnarray}
 \pi_\rho [h_{\tau_{(-1,0)}[F_0(\eta)]}(\tau_{(0,1)}(f))-h_{\tau_{(-1,0)}[F_0(\eta)]}(f)]=c_1(\rho)-1/4:
\end{eqnarray}
in fact, from Proposition \ref{prop:general} and Remark \ref{rem:ag} we have that, if
$\eta\sim\pi_\rho$, then $\tau_{(-1,0)}[F_0(\eta)]$ has the law
$\pi_\rho$ with ``$a$'' and ``$1$'' faces interchanged and the dimer configuration shifted by $-{\bf e_1}$ or, equivalently, has the law of a configuration sampled from $\pi_\rho$ and shifted by $+{\bf e_2}$. Altogether,
\eqref{eq:vrho} follows.
\end{proof}
\begin{Remark}
  If we assumed that $f$ is an even face of type ``$1$'' instead, we
  would get instead of \eqref{eq:vrho}
\begin{eqnarray}
  \label{eq:vrho2}
   v(\rho)=c_3(\rho)-c_4(\rho),
\end{eqnarray}
where $c_3(\rho)$ is the density of horizontal dimers with the $1$
face above and $c_4(\rho)$ the density of vertical dimers with the $1$
face on the left.  Reassuringly, one can prove (via a change of
variables in the integral kernels \eqref{eq:Kast-1} giving the
probabilities $c_i(\rho)$) that
  \begin{eqnarray}
    \label{eq:identita}
    c_1(\rho)-c_2(\rho)=c_3(\rho)-c_4(\rho).
  \end{eqnarray}
  \end{Remark}
  \begin{proof}[Proof of \eqref{eq:flutt1}]
    By stationarity of $\pi_\rho$, the law of $h_k(f)-h_0(f)$ is
    independent of $f$, so we assume without loss of generality that
    $f$ is an even face. Let $\Lambda_\ell$ be the collection of the
    $O(\ell^2)$ even faces  within distance $\ell$ from $f$ and let
    \begin{eqnarray}
      \label{eq:QL}
      Q_{\Lambda_\ell}(k):=\sum_{x\in{\Lambda_\ell}}Q_x(k), \qquad Q_x(k):=h_{k}(x)-h_0(x).
    \end{eqnarray}
    We have clearly
    \begin{eqnarray}
      Q_{{\Lambda_\ell}}(k+1)=      Q_{{\Lambda_\ell}}(k)+K_{{\Lambda_\ell}}(k),\qquad K_{\Lambda_\ell}(k):=
      \sum_{x\in {\Lambda_\ell}}h_{k+1}(x)-h_k(x) 
    \end{eqnarray}
    so that (letting $V_\ell(k):={\rm Var}_{\mathbb P_{\pi_\rho}}(Q_{\Lambda_\ell}(k))$)
    \begin{multline}
      \label{eq:var}
    V_\ell(k+1)-    V_\ell(k)\\=2\mathbb E_{\pi_\rho}\left[
      \left(Q_{\Lambda_\ell}(k)-\mathbb E_{\pi_\rho}Q_{\Lambda_\ell}(k)
      \right)\left(K_{\Lambda_\ell}(k)-\mathbb E_{\pi_\rho}(K_{\Lambda_\ell}(k))
      \right)
                                                                                                                                                \right]+{\rm Var}_{\mathbb P_{\pi_\rho}}(K_{\Lambda_\ell}(k))\\
      \le 2\sqrt{ V_\ell(k)}\sqrt{{\rm Var}_{\mathbb P_{\pi_\rho}}(K_{\Lambda_\ell}(k))}+{\rm Var}_{\mathbb P_{\pi_\rho}}(K_{\Lambda_\ell}(k)).
    \end{multline}
  We will prove in a moment the following
  \begin{Lemma}\label{lemma:vK}
    There exists $C$ such that, for every $k\ge0,\ell\ge1$,
    \begin{eqnarray}
      \label{eq:vK1}
      {\rm Var}_{\mathbb P_{\pi_\rho}}(K_{\Lambda_\ell}(k))={\rm Var}_{\mathbb P_{\pi_\rho}}(K_{\Lambda_\ell}(0)) \le C \ell^2\sigma^2(\ell)
    \end{eqnarray}
    where $\sigma(\cdot)$ is as in \eqref{eq:N}.
  \end{Lemma}
  Given this claim, it is easy to conclude the proof of
  \eqref{eq:flutt1}. Indeed, we see that (with $C_i$ denoting 
  positive constants)
  \begin{eqnarray}
    V_{\ell}(k+1)-V_{\ell}(k)\le C_1
      \sqrt{V_{\ell}(k)} \ell\sigma(\ell)+\ell^2\sigma^2(\ell). 
  \end{eqnarray}
  A simple inductive argument 
 allows to deduce, for $k=\ell$,
  \begin{eqnarray}
    \label{eq:ld}
    V_{k}(k)\le  C_2 k^4 \sigma^2(k).
  \end{eqnarray}
(Just check by induction that $V_\ell(k)\le C_2 k^2 \ell^2\sigma^2(\ell)$  for $k=0,1,\dots,\ell$, if $C_2\ge \max(1,C_1^2/4)$.)

  Now we are ready to prove \eqref{eq:flutt1}.
  Write
  \begin{eqnarray}
   & \mathbb P_{\pi_\rho}(|h_k(f)-h_0(f)-v(\rho)k|\ge u \sigma(k))\\\nonumber
   & =\mathbb P_{\pi_\rho}(|Q_f(k)-\mathbb E_{\pi_\rho}Q_f(k)|\ge u\sigma(k))\\\nonumber
 &   =
   \mathbb P_{\pi_\rho}(|Q_f(k)-\mathbb E_{\pi_\rho}Q_f(k)|\ge u\sigma(k);|Q_{\Lambda_k}(k)-\mathbb E_{\pi_\rho}Q_{\Lambda_k}(k)|\le  \sqrt{u}k^2\sigma(k))
    +o(1)
  \end{eqnarray}
where $o(1)$ tends to zero as $u\to\infty$ uniformly in $k$, by Tchebyshev's inequality, thanks to \eqref{eq:ld}.
On the other hand
\begin{gather}
  Q_{\Lambda_k}(k)-  \mathbb E_{\pi_\rho} Q_{\Lambda_k}(k)=-A_k(0)+A_k(k)+\tilde A_k\\
A_\ell(k)
  :=\sum_{x\in\Lambda_\ell}[h_k(x)-h_{k}(f)-\pi_\rho(h_k(x)-h_k({f}))]\\
\tilde A_\ell:=
|\Lambda_\ell|[Q_{f}(k)-\mathbb E_{\pi_\rho} Q_{f}(k)].
\label{eq:qf}
\end{gather}
In the second line, $\pi_\rho(h_k(x)-h_k({f}))$ is actually time-independent.
The law of $A_\ell(k)$ is time-independent by stationarity of $\pi_\rho$ and moreover:
\begin{Lemma}
  \label{lemma:dh}
  One has 
\begin{gather}
 \label{tec3}
 \pi_\rho\left[A_k(0)^2\right]=
     O(k^4\sigma^2( k)).
\end{gather}
\end{Lemma}
Therefore, by Tchebyshev's inequality,
\begin{gather}
  \label{eventicchio} |A_k(0)|,|A_k(k)|\le 
{\sqrt u} k^{2}\sigma(k)
  ,
\end{gather}
with probability $1+o(1)$ as $u\to\infty$.
Finally, we note from \eqref{eq:qf} that if event \eqref{eventicchio} holds and 
at the same time
$|Q_{\Lambda_k}(k)-  \mathbb E_{\pi_\rho}
Q_{\Lambda_k}(k)|\le
\sqrt{u}k^2\sigma(k)
 $, since $|\Lambda_k|$ grows proportionally to $k^2$ one cannot
have $|Q_f(k)-\mathbb E_{\pi_\rho}Q_f(k)|\ge u\sigma(k)$ for $u$ large. Eq. \eqref{eq:flutt1} is then proven.
  \end{proof}

  \begin{proof}[Proof of Lemma \ref{lemma:vK}]
    Given that the sum in the definition of $K_{\Lambda_\ell}(k)$ runs over even faces, we apply \eqref{eq:perplesso} to write
    \begin{gather*}
      K_{\Lambda_\ell}(k)=K^{(1)}_{\Lambda_\ell}(k)+K^{(2)}_{\Lambda_\ell}(k)\\
      K^{(j)}_{\Lambda_\ell}(k)=\sum_{x\in\Lambda_\ell}r^{(j)}_x(k)\\
      r^{(1)}_x(k)=h_{\tau_{(-1,0)}[F_0(\eta(k))]}(\tau_{(0,1)}(x))-h_{\tau_{(-1,0)}[F_0(\eta(k))]}(x)\\
      r^{(2)}_x(k)=h_{\eta(k)}(\tau_{(1,0)}(x))-h_{\eta(k)}(x)
    \end{gather*}
    and by Cauchy-Schwarz it is sufficient to upper bound  the variances of $K^{(j)}_{\Lambda_\ell}(k),j=1,2$
    separately. By definition of the map $F_0$ and of the height function, 
    \begin{eqnarray}
      \label{eq:r12}
      r^{(1)}_x(k)=1_{e'_x\in [\tau_{(-1,0)}F_0(\eta(k))]}-1/4,\qquad r^{(2)}_x(k)=1/4-1_{e_x\in\eta(k)}
    \end{eqnarray}
    with $e_x$ (resp. $e'_x$) the vertical edge connecting the
    top-right and the bottom-right (resp. top-right and top-left)
    vertices of face $x$.  Since the laws of $\eta(k)$ and
    $F_0(\eta(k))$ are stationary, the variance of
    $K^{(j)}_{\Lambda_\ell}(k),j=1,2$ are independent of $k$. One has
    \begin{eqnarray}
      \label{eq:varK2}
	    {\rm Var}_{\mathbb{P}_{\pi_\rho}}(K^{(2)}_{\Lambda_\ell}(k))=\sum_{x,y\in\Lambda_\ell}\pi_\rho(e_x\in\eta;e_y\in\eta)
    \end{eqnarray}
    with $\pi(A;B)$ denoting the covariance of $A,B$. It is known
    \cite{KOS03} that dimer-dimer correlations decay like the inverse
    distance square if $\pi_\rho$ is a rough phase (i.e. $\rho\in\mathcal L$), and exponentially
    fast if $\pi_\rho$ is a smooth phase. Then, inequality
    \eqref{eq:vK1} for $K^{(2)}_{\Lambda_\ell}(k)$ immediately
    follows. The argument for $K^{(1)}_{\Lambda_\ell}(k)$ is
    essentially identical.
  \end{proof}

  \begin{proof}[Proof of Lemma \ref{lemma:dh}]
    This is an immediate consequence of the known fact that, if
    $\pi_\rho$ is a rough phase (as is the case for $\rho\ne0$ or
	  $\rho=0,a=1$) then the variance ${\rm Var}_{{\pi_\rho}}(h(x)-h(y))$
    grows proportionally to $\log |x-y|$ when $|x-y|\to\infty$, while
    if $\pi_\rho$ is a smooth phase then the variance is uniformly
    bounded in $x,y$ \cite{KOS03}.
  \end{proof}
\section{Proof of Theorem \ref{th:v}}
\label{sec:compv}
It is easy to prove \eqref{eq:v} when $|\rho_1|=|\rho_2|$ (i.e. when
$\rho_+\rho_-=0$), by symmetry considerations.  Indeed, a reflection
across a line passing through the center of a face and forming an angle $-\pi/4$ with respect to the horizontal
axis maps $\pi_{(\rho_1,\rho_2)}$ to $\pi_{(\rho_2,\rho_1)}$ and, from
definition \eqref{eq:vrho}, changes sign of the speed
$v(\rho)=c_1(\rho)-c_2(\rho)$ since it exchanges edges $e_1$ and $e_2$ in \eqref{eq:c1c2}. In other words, \eqref{eq:simmv} holds. 
 Similarly, a reflection
across a line forming an angle $\pi/4$ with respect to the horizontal
axis maps $\pi_{(\rho_1,\rho_2)}$ to $\pi_{(-\rho_2,-\rho_1)}$
and implies \eqref{eq:simmv1} (see caption of Fig. \ref{fig:symm}). To obtain \eqref{eq:simmv2}, take reflections w.r.t. horizontal/vertical lines throught the center of a face.
\begin{figure}
\begin{center}
\includegraphics[height=4cm]{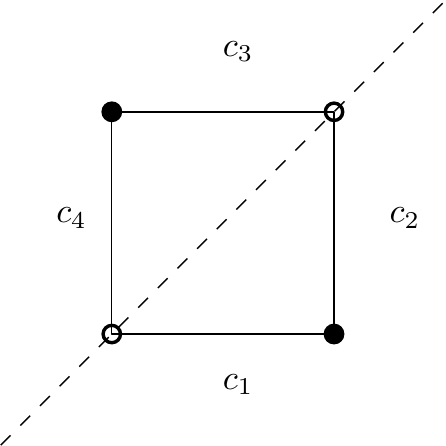} 
\end{center}
\caption{An even face of type $a$, with $c_1,\dots,c_4$ denoting the densities of the four edges on its boundary, so that $v=c_1-c_2$, $\rho_1=c_4-c_2,\rho_2=c_3-c_1$. After reflection across the dashed line, $c_1$ turns into $c_4$ and $c_2$ into $c_3$. Therefore, the slope becomes $(-\rho_2,-\rho_1)$ and the speed becomes $c_4-c_3=\rho_1-\rho_2-v(\rho_1,\rho_2)$.} 
\label{fig:symm}
\end{figure}
Note that indeed the r.h.s. of \eqref{eq:v} satisfies symmetries
\eqref{eq:simmv}-\eqref{eq:simmv2}.

As a consequence of the symmetries, the speed
must vanish when $\rho_1=\rho_2$ (i.e. when $\rho_-=0$) and must equal
$\rho_1$ when $\rho_2=-\rho_1$, i.e. when $\rho_+=0$, as  the r.h.s. of \eqref{eq:v} indeed does.

In the rest of the section, we will therefore assume that $|\rho_1|\ne|\rho_2|$.
\subsection{Kasteleyn matrix} 

\label{sec:Kmatr} The speed of growth $v(\rho)$ is given by the difference between the
probabilities of two events. We begin by recalling how to express
probabilities of local events via the Kasteleyn matrix  and
how to rewrite its matrix elements as single integrals in the complex
plane, as was done in~\cite{CJ16}.   We adopt a similar coordinate system to the one used in~\cite{CJ16} but we have chosen to interchange the white and black vertices, so that we can keep the same height change conventions from a previous paper~\cite{chhita2017speed}.

 Namely, with reference to Fig.~\ref{fig:fundamental}, where the graph has been rotated $45$ degrees clockwise, we
assign to each vertex in $\mathbb Z^2$ coordinates $x=(x_1,x_2)$ with
$x_1+x_2\in 2\mathbb Z+1$. Then, the set of black vertices is
$$
\mathtt{B}= \{(x_1,x_2) \in \mathbb{Z}^2: x_1\!\!\! \mod 2=0, x_2 \!\!\!\mod 2=1\}
$$
and the set of white vertices is
$$
\mathtt{W}= \{(y_1,y_2) \in \mathbb{Z}^2: y_1 \!\!\!\mod 2=1, y_2\!\!\! \mod 2=0\}.
$$
For $i \in\{0,1\}$, we also let
$$
\mathtt{B}_i= \{(x_1,x_2) \in \mathtt{B}: x_1+x_2\!\!\! \mod4 =2i+1 \}
$$
and
$$
\mathtt{W}_i= \{(y_1,y_2) \in \mathtt{W}: y_1+y_2 \!\!\!\mod4 =2i+1 \}.
$$
The fundamental domain of the two-periodic weighting is given by a 2
by 2 block consisting of a vertex from each
$\mathtt{W}_0, \mathtt{W}_1, \mathtt{B}_0$ and $\mathtt{B}_1$.  We
suppose that if the vertex $b \in \mathtt{B}_0$ is in the fundamental
domain, then so are the vertices $b+ \hat e_2 \in \mathtt{W}_0$,
$b+\hat e_1 \in \mathtt{W}_1$, and $b+\hat e_1 +\hat e_2\in \mathtt{B}_1$, with
$\hat e_1 = (1,1)$ and $\hat e_2=(-1,1)$.  We impose that the edges inside the
fundamental domain have weight $a$ and the edges crossing to another
fundamental domain have weight $1$; see Fig.~\ref{fig:fundamental}.

\begin{figure}
\begin{center}
\includegraphics[height=7cm]{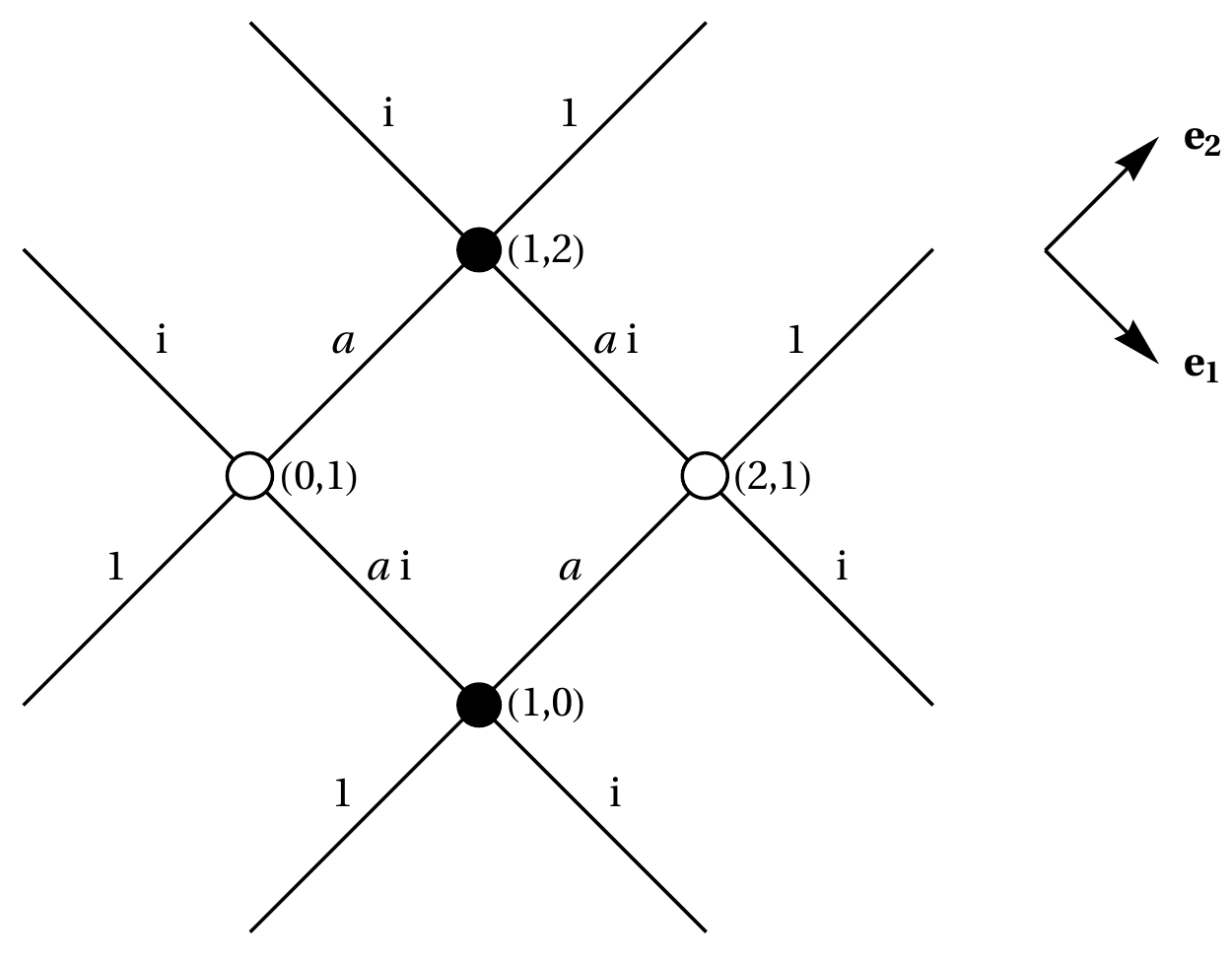} 
\end{center}
\caption{
  The fundamental domain and the coordinates of its four
  vertices, where the coordinate axes (not drawn) are horizontal/vertical, with the origin in the center of the bottom-left face.  The vectors
  ${\bf e_1},{\bf e_2}$ are instead the 45 degree rotations of the
  vectors introduced in Section \ref{sec:actual}, and they denote the two direction
  w.r.t. which we compute the slope. }
\label{fig:fundamental}
\end{figure}

The Kasteleyn matrix $\mathbb K$ is a weighted version of the adjacency matrix of $\mathbb Z^2$. The matrix element 
$\mathbb K(y,x)$ is zero if $x\in \mathtt{B} $ and $y\in \mathtt{W}$ are not neighbors. If they are neighbors, then 
$\mathbb K(y,x)$ equals the weight ($a$ or $1$) of the edge from $x$ to $y$, times i (the imaginary unit) if the edge is parallel to $\hat e_2$. The non-zero values of the matrix elements are therefore $a,\textrm{i}a,1,\textrm{i}$.

Given positive real numbers $r_1,r_2$, we introduce the ``inverse Kasteleyn matrix'' $\mathbb K^{-1}_{r_1,r_2}$ as follows.
Define first the $2\times 2$ matrix  
\[
K(z,w) = \left(  \begin{array}{cc}
\mathrm{i}(a+w^{-1}) & a+z \\
a+z^{-1} & \mathrm{i}(a+w) 
\end{array}
\right)
\]
 and the \emph{characteristic polynomial}
  \begin{eqnarray}
    \label{eq:chpol}
    P(z,w)=-\det K(z,w)=a\left(\frac2c+w+\frac1w+z+\frac1z\right).   
  \end{eqnarray}
  Let $x \in \mathtt{B}_{\e_1}$ and $y \in \mathtt{W}_{\e_2}$ and let
  $(u,v)\in\mathbb Z^2$ be such that the unique fundamental domain
  containing $x$ (resp. $y$) has black vertex in $\mathtt{B_0}$ with
  coordinates $X$ (resp. $Y$) satisfying
  $Y-X=2(u \hat e_1+v \hat e_2)$. Then, we let
\begin{equation}
\label{eq:Kast-1}
\mathbb{K}^{-1}_{r_1,r_2} (x, y )=\frac{1}{(2 \pi \mathrm{i})} \int_{\Gamma_{r_1}}\frac{dz}{z} \int_{\Gamma_{r_2}} \frac{dw}{w}  \left( K(z,w) \right)^{-1}_{\e_1+1,\e_2+1} z^u w^v
\end{equation}
where the integrals are taken in the complex plane and $\Gamma_{r}$ is a contour of radius $r$ around the origin, oriented anti-clockwise.

The link between the Gibbs measure $\pi_\rho$ and the Kasteleyn matrix
is as follows \cite{KOS03}. 
Whenever the slope $\rho$ corresponds to a rough phase, i.e. $\rho\in\mathcal L$ (recall \eqref{eq:calL})
there exists a unique choice
$r=(r_1,r_2)=r(\rho)$ (magnetic coordinates) such that, given any integer $n$ and edges 
$e_i,i\le n$ with black/white vertices of coordinates $x_i$ and $y_i$ respectively, one has
\begin{eqnarray}
  \label{eq:statistics}
  \pi_\rho(e_1,\dots,e_n\in \eta)=\left(\prod_{i=1}^n \mathbb K(y_i,x_i)\right) \det\{\mathbb K^{-1}_{r_1,r_2}(x_i,y_j)\}_{i,j\le n}.
\end{eqnarray}
The image of the curve $P(z,w)=0$ in $\mathbb{C}^2$ under the map
$(z,w) \mapsto (\log |z|, \log |w|)$ is called the \emph{amoeba} of
$P$, often denoted by $\mathbb{A}(P)$. The set
$\mathcal B=\{r:r=r(\rho) \text{ for some } \rho\in\mathcal L\}$ is
the interior component of the amoeba and the correspondence
$r\in\mathcal B\leftrightarrow \rho(r)\in\mathcal L$ is a bijection
\cite{KOS03}.  
In the rest of this section, we assume that
$r\in\mathcal B$. Note that this is not a restriction since $\rho=0$ has already been taken into account.

Both the speed $v(\rho)$ (through \eqref{eq:vrho}) and the slope $\rho$ (through \eqref{eq:hxy}) can be expressed as linear combinations of probabilities as in the l.h.s. of \eqref{eq:statistics}, with $n=1$.
Before working out the explicit expressions, let us recall a few known facts \cite{CJ16} about $\mathbb K^{-1}_{r_1,r_2}$.
Set
$$
h(\e_1,\e_2)=\e_1 (1-\e_2) + \e_2 (1-\e_1)
$$
and 
$$
\tilde{c}(u_1,u_2)=2(1+a^2)+a(u_1+u_1^{-1})(u_2+u_2^{-1}).
$$
Note that $\tilde{c}(u_1,u_2)=P(u_1/u_2,u_1 u_2)$.
We next state the following lemma without proof from~\cite{CJ16}:
\begin{Lemma}[Lemma 4.3 from~\cite{CJ16}] \label{CJ16Lem4.3}
 For $x=(x_1,x_2)\in\mathtt{B}_{\e_1}$ and $y=(y_1,y_2)\in\mathtt{W}_{\e_2}$ with $\e_1, \e_2\in \{0,1\}$,
\begin{equation}
\mathbb{K}^{-1}_{r_1,r_2} (x,y)= \frac{-\mathrm{i}^{1+h(\e_1,\e_2)}}{(2 \pi \mathrm{i})^2}  \int_{\Gamma_{R_1}} \frac{du_1}{u_1}  \int_{\Gamma_{R_2}} \frac{du_2}{u_2} \frac{ a^{\e_2} u_2^{h(\e_1,\e_2)-1}+a^{1-\e_2}u_1 u_2^{-h(\e_1,\e_2)}}{\tilde{c}(u_1,u_2) u_1^{\frac{x_1-y_1+1}{2}} u_2^{\frac{y_2-x_2-1}{2}}}
\end{equation}
where  $R_1=\sqrt{r_1/r_2}$ and $R_2 = 1/\sqrt{r_1 r_2}$.
\end{Lemma}
The proof of the above lemma, see~\cite{CJ16}, follows from a change of variables, which also explains the shift in contours of integration. 
The following are from~\cite{CJ16} and are useful for later computations. Define for $w \in\mathbb{C} \backslash \mathrm{i}[-\sqrt{2c},\sqrt{2c}]$ 
\begin{eqnarray}
  \label{eq:sqrt}
\sqrt{w^2+2c} =e^{\frac{1}{2} \log (w +\mathrm{i}\sqrt{2c})+ \frac{1}{2} \log (w -\mathrm{i}\sqrt{2c})}  
\end{eqnarray}
where the logarithm takes arguments in $(-\pi/2,3\pi/2)$. Write $\sqrt{1/w^2+2c}$ for the same function evaluated at $1/w$.
Set
$$
G(w)=\frac{1}{\sqrt{2c}} (w-\sqrt{w^2+2c}).
$$
Note that, since $c\le 1/2$, the branch cut $\mathrm{i}[-\sqrt{2c},\sqrt{2c}]$  is included in $[-\mathrm{i},\mathrm{i}]$. The choice in branch cut also gives that 
\begin{eqnarray}
  \label{eq:G-G}
 \sqrt{w^2+2c}=-\sqrt{(-w)^2+2c}.
\end{eqnarray}
Let us note the following, for later use:
\begin{Remark}
  \label{rem:zeta}
  We have 
  \begin{eqnarray}
    \label{eq:argsqrtz}
    \sqrt{z^2+2c}=\left\{
    \begin{array}{ll}
      \sqrt{x^2+2c}& \text{if $z=x\in\mathbb R^+$}\\
      \sqrt{2c-y^2}& \text{if $z=iy+0^+, -\sqrt{2c}< y<\sqrt{2c}$}\\
      i \sqrt{y^2-2c} & \text{if $z=iy+0^+,y>\sqrt{2c}$}\\
      -i \sqrt{y^2-2c}&\text{if $z=iy+0^+,y<-\sqrt{2c}$}
    \end{array}
               \right..
  \end{eqnarray}
  If $z\to i\sqrt{2c}+0^+$ (resp.  $z\to -i\sqrt{2c}+0^+$), then
  any limit point of $\arg\sqrt{z^2+2c} $ is in $[0, \pi/2]$ (resp. in $[- \pi/2,0]$).
\end{Remark}

The transformation $u=G(\omega)$ is a map between $|u|<1$ and the
whole plane (without the cut) in $\omega$. For
$u=Re^{ \mathrm{i} \theta}$ with $R<1$, the inverse of this
transformation is given by
\begin{equation}\label{speed:cov}
\omega = \sqrt{\frac{c}{2}} (R e^{\mathrm{i} \theta}- R^{-1} e^{-\mathrm{i} \theta}).
\end{equation}
Moreover, when $R<1$ and $\theta$ runs from $0$ to $2\pi$, $\omega$ in
\eqref{speed:cov} runs (clockwise) over an ellipse, that does not cross the cut. We   denote by $\gamma_R$ the ellipse, oriented \emph{anti-clockwise}.

For $x= (x_1,x_2) \in \mathtt{B}_{\e_1}$ and $y=(y_1,y_2) \in\mathtt{W}_{\e_2}$ define for $i \in \{1,2\}$
\begin{equation}
\begin{split}
k_i&= \frac{x_2-y_2+(-1)^i}{2}-(-1)^i h(\e_1,\e_2)  \\
l_i&=\frac{y_1-x_1+(-1)^i}{2}.
\end{split}
\end{equation}
Define for two contours $\gamma$ and $\gamma'$ (that avoid the branch cut) and $k,l \in \mathbb{Z}$
\begin{equation}
  \begin{split}
    \label{eq:2contorni}
\tilde{D}_{\gamma,\gamma'}(k,l)=\frac{\mathrm{i}^{-k-l}}{(2 \pi \mathrm{i})^22(1+a^2)} \int_\gamma d \omega_1 \int_{\gamma'}d \omega_2  \frac{ G (\omega_1)^l G( \omega_2)^{k} }{(1-\omega_1 \omega_2 ) \sqrt{\omega_1^2+2c} \sqrt{\omega_2^2+2c}}
\end{split}
\end{equation}
and also define
\begin{equation}
\begin{split}
D_{\gamma,\gamma'}(x,y)=-\mathrm{i}^{1+h(\e_1,\e_2)} (a^{\e_2} \tilde{D}_{\gamma,\gamma'}(k_1,l_1)+a^{1-\e_2} \tilde{D}_{\gamma,\gamma'}(k_2,l_2)).
\end{split}
\end{equation}
We state the following lemma without proof from~\cite{CJ16}.
\begin{Lemma}[Lemma 4.4 in~\cite{CJ16}]\label{CJ16Lem4.4}
Let $R_1= \sqrt{r_1/r_2}$ and $R_2=1/\sqrt{r_1 r_2}$  with $R_1, R_2<1$. Then, we have 
\begin{equation}
\mathbb{K}^{-1}_{r_1,r_2}(x,y) = D_{\gamma_{R_1},\gamma_{R_2}} (x,y).
\end{equation}
\end{Lemma}
The proof of this lemma involves using the change of variables $u=G(\omega)$  for each variable $u_1$ and $u_2$.
Note that the contours $\gamma_{R_1},\gamma_{R_2}$ encircle the branch cut $\mathrm{i}[-\sqrt{2c},\sqrt{2c}]$ of $G(\omega_j),j=1,2$.
\begin{Remark}
  \label{rem:genth}
The general theory of \cite{KOS03} implies that, whenever $r\in\mathcal B$ (equivalently, when 
$r=r(\rho)$ with 
$\rho\in\mathcal L$), the
polynomial $P(r_1 z,r_2 w)$ has two and only two zeros on
$\mathbb T=\{(z,w)\in\mathbb C^2,|z|=|w|=1\}$, that are simple and conjugate, i.e. of the form $(z_0, w_0) $ and $(\overline{z_0}, \overline{w_0})$. \end{Remark}
We have the following
lemma which is a culmination and extension of Lemmas 3.3 and 4.5 
from~\cite{CJ16}.
\begin{Lemma}\label{lem:existsomega}
 Let 
  $r\in\mathcal B$, let  $R_1, R_2$ be defined as in Lemma~\ref{CJ16Lem4.4} as functions
  of $r$ and assume $R_1,R_2<1$.
Then,  there exists a unique
  $\omega_c \in \mathbb{R}_{>0} \times \mathrm{i} \mathbb{R}_{>0}$
  such that
\begin{equation}
\begin{split}
\tilde{D}_{\gamma_{R_1},\gamma_{R_2}} (k,l)&= 
\frac{\mathrm{i}^{-k-l}}{4\pi \mathrm{i}(1+a^2)} \left(\int_{\overline{\omega}_c}^{\omega_c} d \omega+\int_{-\overline{\omega}_c}^{-\omega_c} d \omega\right) \frac{ G(\omega^{-1})^l G(\omega)^k}{\omega \sqrt{\omega^2+2c} \sqrt{\omega^{-2}+2c}} \\
&+\tilde{S}(k,l)
\end{split}
\end{equation}
where the integration path from $\overline{\omega_c}$ to $\omega_c$ (resp. from
  $-\overline{\omega_c}$ to $-\omega_c$) stays in the half-plane with positive (resp. negative) real part, 
\begin{equation}
\tilde{S}(k,l)= \left\{ \begin{array}{ll}
0 & \mbox{if } k \geq 0 \mbox{ or } l \geq 0 \\
a^{-1} & \mbox{if }k=l=-1 \\ 
\frac{\mathrm{i}^{-k-l}}{(2 \pi \mathrm{i})^2 2(1+a^2)} \int_{\Gamma_{R}} d \omega_1 \int_{\Gamma_R} d\omega_2 \frac{ G(\omega_1)^l G(\omega_2)^k}{(1-\omega_1 \omega_2) \sqrt{\omega_1^2+2c} \sqrt{\omega_2^2+2c}} & \mbox{otherwise,}\end{array} \right.
\end{equation}
and $R>1$.
\end{Lemma}

\begin{proof}[Proof of Lemma \ref{lem:existsomega}] The function in the r.h.s. of \eqref{eq:2contorni} has a pole whenever $\omega_1=1/\omega_2$. We show first of all:
  \begin{Claim}
    \label{claim:4}
    There exist four distinct values of $\omega_1\in \gamma_{R_1}$
    such that $\omega_2:=1/\omega_1\in \gamma_{R_2}$ and there is
    exactly one of them for each of the four quadrants of
    $\mathbb C\setminus \{\mathbb R\cup \mathrm{ i}\mathbb R\}$. We
    call by convention $\omega_c$ the one that is in
    $ \mathbb{R}_{>0} \times \mathrm{i} \mathbb{R}_{>0}$, the others
    being $-\omega_c,\overline{\omega_c},-\overline{\omega_c}$.
  \end{Claim}
  \begin{proof}
    Since each $\omega_i,i=1,2$ runs along the
  clockwise-oriented ellipse $\gamma_{R_i}$ defined by
  \eqref{speed:cov},  we have
\begin{equation}
\begin{split}
\omega_1 \omega_2&=\frac{c}{2} \left(R_1 e^{\mathrm{i} \theta_1} - R_1^{-1} e^{-\mathrm{i} \theta_1} \right)\left(R_2 e^{\mathrm{i} \theta_2} - R_2^{-1} e^{-\mathrm{i} \theta_2} \right) \\
&= \frac{c}{2} \left( \frac{1}{r_2} e^{\mathrm{i}(\theta_1+\theta_2)} -\frac{1}{r_1} e^{-\mathrm{i}(\theta_1-\theta_2)}-r_1e^{\mathrm{i}(\theta_1-\theta_2)} +r_2 e^{-\mathrm{i}(\theta_1+\theta_2)} \right).
\end{split}
\end{equation}
	Set $u_1=\mathrm{i} e^{\mathrm{i} \theta_1}$ and $u_2=\mathrm{i}e^{\mathrm{i} \theta_2}$. Then,
\begin{equation}
  \begin{split}
\omega_1 \omega_2&=\frac{c}{2} \left(-\frac{1}{r_2} u_1 u_2-r_2 \frac{1}{u_1 u_2} -r_1 \frac{u_1}{u_2} - \frac{1}{r_1} \frac{u_2}{u_1} \right).
\end{split}
\end{equation}
Set $w=1/(u_1u_2)$ and $z=u_1/u_2$ and note $|z|=|w|=1$. This gives
\begin{equation}
\begin{split}
\omega_1 \omega_2= \frac{c}{2} \left(- {w}{r_2 }  -\frac1{r_2 w} -r_1 z -\frac{1}{r_1 z} \right) .
\end{split}
\end{equation} Comparing with \eqref{eq:chpol} 
we conclude that $1 - \omega_1 \omega_2=0$ is equivalent to having $P(z r_1, w r_2)=0$. 
Recall from Remark \ref{rem:genth} that $z_0,w_0$ are not both real.
We also have that they cannot be both purely imaginary, since it is clear
from \eqref{eq:chpol} that $P(z_0r_1,w_0r_2)$ would not be zero in that case. Therefore, we can assume that at least one
among $z_0/w_0$ and $z_0w_0$ is not real, and assume that the former is the case (in the other case, the following
argument works the same exchanging the roles of $u_1$ and $u_2$). Recalling the definitions of $z,w$ in terms of $u_1,u_2$
we see that the condition $\omega_1\omega_2=1$ implies that $u_1$ can have only one of the four possible values
$\pm\sqrt{z_0/w_0},\pm \sqrt{\overline{z_0}/\overline{w_0}}$. The four values are distinct and there is one of them
in each of the four quadrants of $\mathbb C\setminus \{\mathbb R\cup \mathrm{i}\mathbb R\}$. As a consequence, Claim
\ref{claim:4} immediately follows.
     \end{proof}
     We resume the proof of Lemma \ref{lem:existsomega}.  Via the
     change of variables $\omega_1\mapsto \omega_1^{-1}$,
\begin{equation}
\begin{split}\label{eq:contourgoodform}
\tilde{D}_{\gamma_{R_1},\gamma_{R_2}}(k,l)&=\frac{\mathrm{i}^{-k-l}}{(2 \pi \mathrm{i})^22(1+a^2)} \int_{\tilde{\gamma}_{R_1}} \frac{d \omega_1}{\omega_1} \int_{\gamma_{R_2}}d \omega_2  \frac{ G (\omega_1^{-1})^l G( \omega_2)^{k} }{(\omega_1 -\omega_2 ) \sqrt{\omega_1^{-2}+2c} \sqrt{\omega_2^2+2c}}
\end{split}
\end{equation}
where $\tilde{\gamma}_{R_1}$ represents the contour from
$\gamma_{R_1}$ under the image $\omega \mapsto 1/\omega$,
oriented anti-clockwise.  Claim \ref{claim:4} indicates that
$\tilde{\gamma}_{R_1}$ intersects $\gamma_{R_2}$ at four points,
$\omega_c, -\omega_c, \overline{\omega}_c$, and
$-\overline{\omega}_c$. Since $R_1, R_2<1$, the contour
$\tilde{\gamma}_{R_1}$ is outside of $\gamma_{R_2}$ on the real axis
and inside $\gamma_{R_2}$ on the imaginary axis. Note also that the
branch cut 
$\mathrm{i}(-\infty,-1/\sqrt{2c}]\cup \mathrm {i}[1/\sqrt{2c},\infty)$
of $G(1/\cdot)$ is outside the contour $\tilde\gamma_{R_1}$.  Fixing
$\omega_2$ we integrate over $\omega_1$, by deforming the contour
$\tilde{\gamma}_{R_1}$ to a circle $\Gamma_{1/R}$ with small
$1/R<R_2<1$.  The deformation crosses a pole $\omega_1=\omega_2$ if
and only if $\omega_2$ is in the portion of $\gamma_{R_2}$ from
$-\overline{\omega_c}$ to $\omega_c$ or from $\overline{\omega_c}$ to
$-\omega_c$. Fig.~\ref{fig:contours} shows these deformations.
\begin{figure}
\begin{center}
\includegraphics[width=4cm]{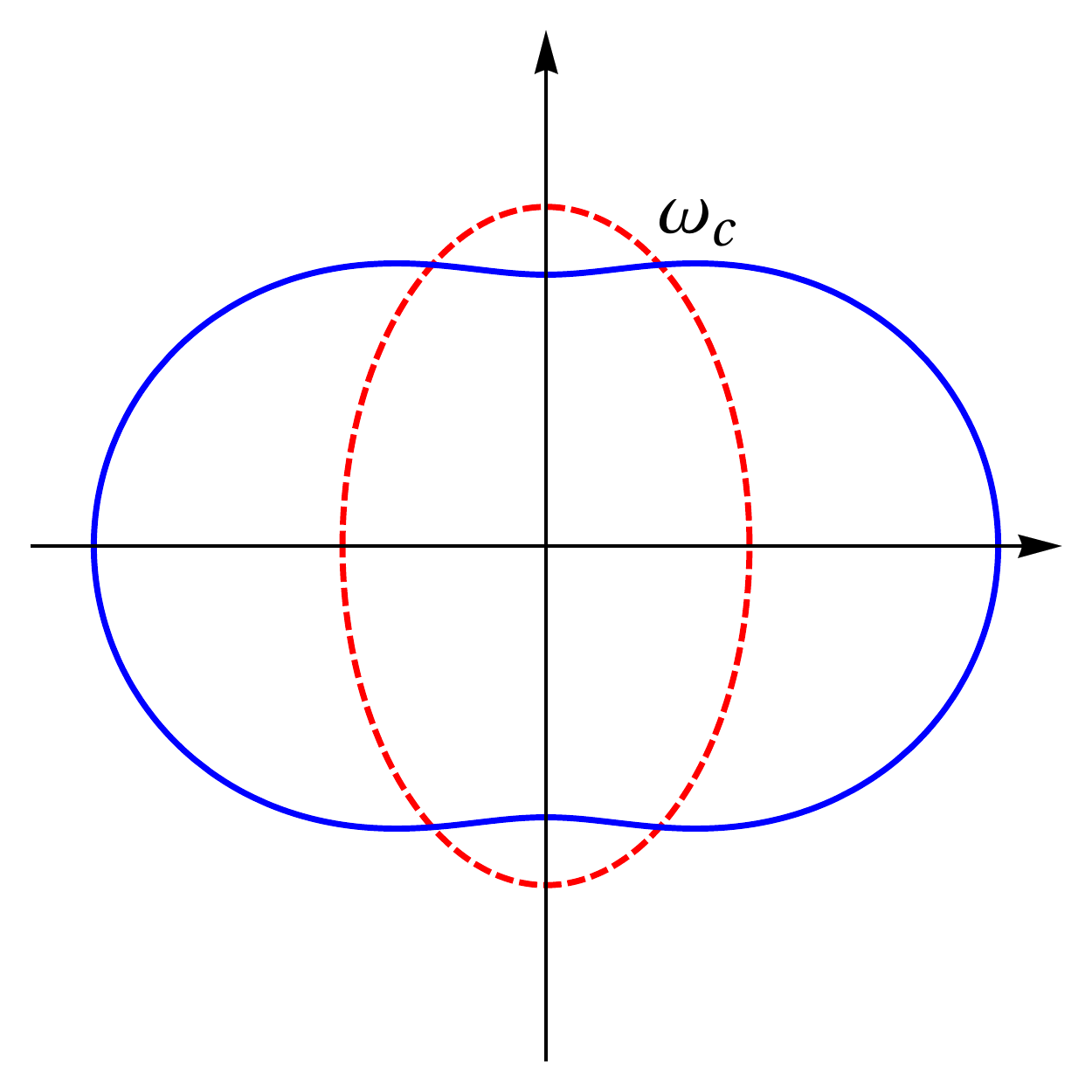}
\hspace{2mm}
\includegraphics[width=4cm]{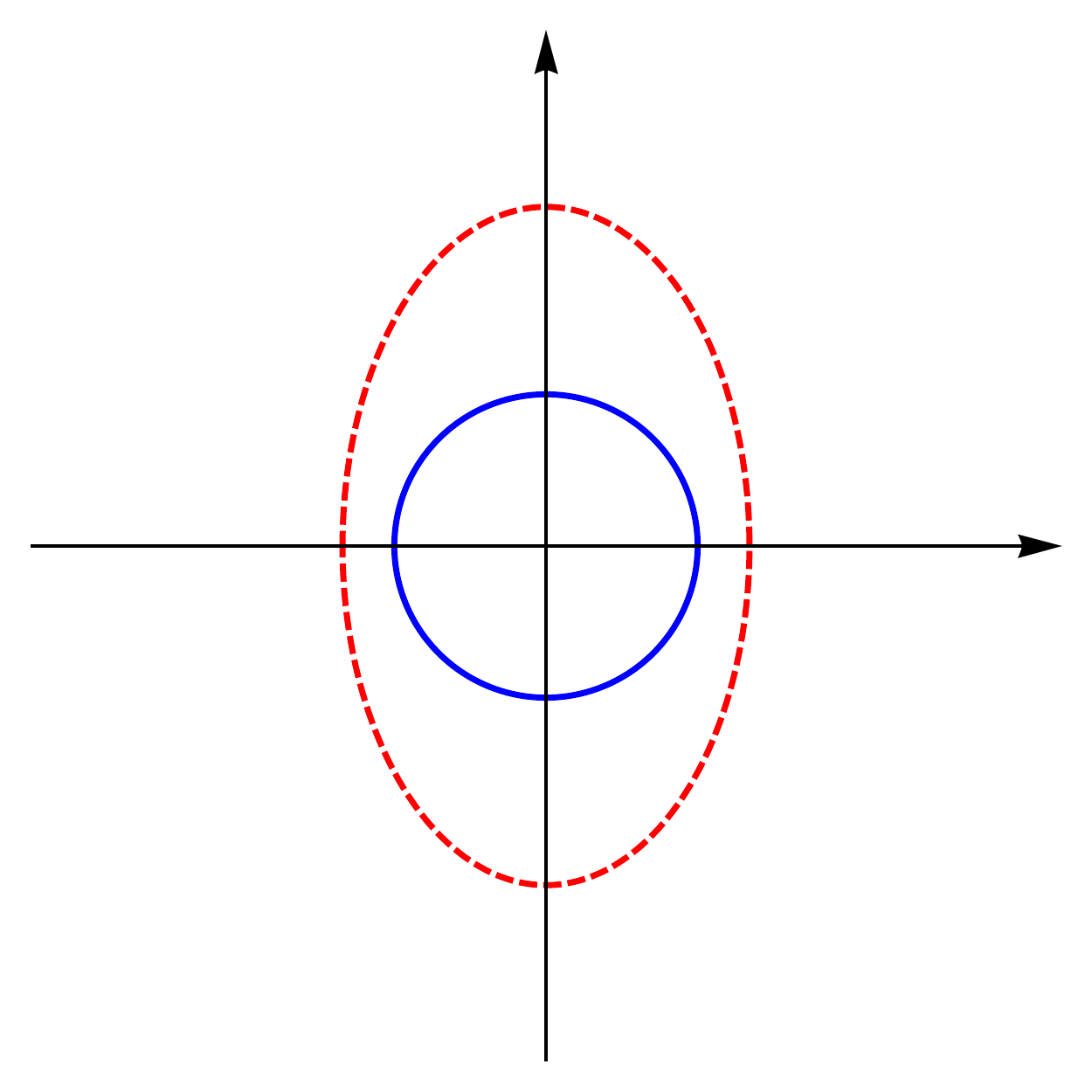}
\hspace{2mm}
\includegraphics[width=4cm]{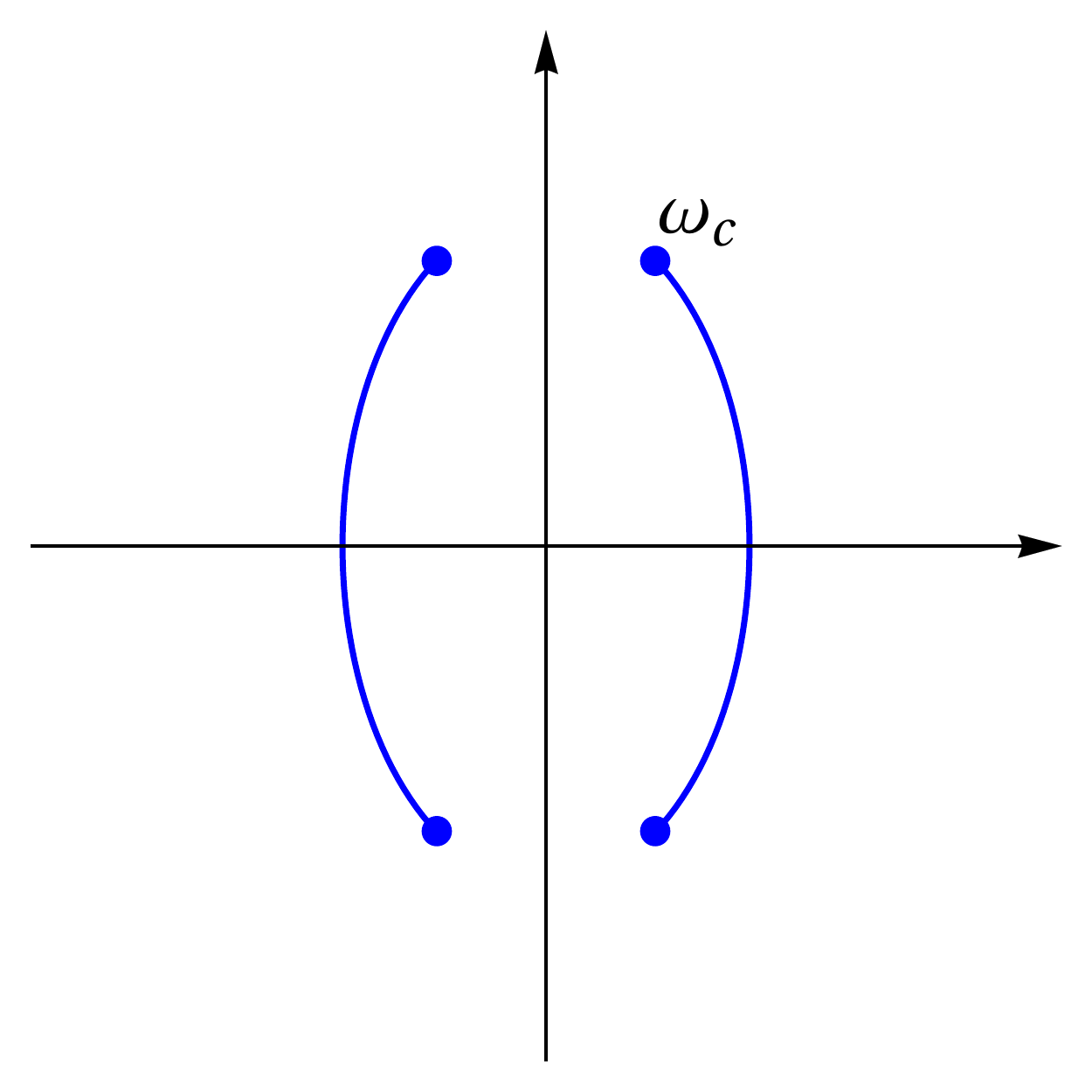}
\caption{The left figure shows the contour in $\tilde{\gamma}_{R_1}$
  in blue and the contour $\gamma_{R_2}$ in red (dashed). The
  deformation of moving the blue contour inside the red contour, shown
  in the center figure, picks up a single contour integral from the
  simple pole contribution on the r.h.s of the integrand
  in~\eqref{eq:contourgoodform}, as shown in the right figure.}
\label{fig:contours}
\end{center}
\end{figure}

 The integral becomes then
\begin{equation}
\begin{split}
  \tilde{D}_{\gamma_{R_1},\gamma_{R_1}}(k,l)&=\frac{\mathrm{i}^{-k-l}}{4 \pi \mathrm{i}\,(1+a^2)} \left( \int_{\overline{\omega}_c}^{\omega_c} +  \int_{-\overline{\omega}_c}^{-\omega_c}  \right) \frac{d \omega}{\omega}  \frac{ G (\omega^{-1})^l G( \omega)^{k} }{ \sqrt{\omega^{-2}+2c} \sqrt{\omega^2+2c}} \\
  &+\frac{\mathrm{i}^{-k-l}}{(2 \pi \mathrm{i})^22(1+a^2)}
  \int_{\Gamma_{1/R}} \frac{d \omega_1}{\omega_1} \int_{\gamma_{R_2}}d
  \omega_2 \frac{ G (\omega_1^{-1})^l G( \omega_2)^{k} }{(\omega_1
    -\omega_2 ) \sqrt{\omega_1^{-2}+2c} \sqrt{\omega_2^2+2c}}
\end{split}
\end{equation}
where the first line is the contribution from the pole.
For the double integral in the above equation, we can deform $\gamma_{R_2}$ to $\Gamma_{R}$ because $1/R<1<R$, so that the branch cut of $G(\cdot)$ remains inside the contour. We can then apply the change of variables $\omega_1 \mapsto 1/\omega_1$  and so the double integral becomes
\begin{equation}
  \frac{\mathrm{i}^{-k-l}}{(2 \pi \mathrm{i})^22(1+a^2)} \int_{\Gamma_{R}} {d \omega_1} \int_{\Gamma_{R}}d \omega_2  \frac{ G (\omega_1)^l G( \omega_2)^{k} }{(1-\omega_1 \omega_2 ) \sqrt{\omega_1^{2}+2c} \sqrt{\omega_2^2+2c}} ,
\end{equation}
where the orientation of $\Gamma_{R}$ is as usual anti-clockwise.  To show that the above integral is equal to
$\tilde{S}(k,l)$, notice that $G(\omega)$ behaves like
$\sqrt{c}/(\sqrt{2}\omega)$ as $|\omega| \to \infty$ and when
$k\geq 0$ or $l \geq 0$, we can push one or the other contour through infinity to
get that the integral is zero.  Also when $l=k=-1$, we can compute the
residue at infinity which gives $1/a$ and hence the double integral in
the above equation is equal to $\tilde{S}(k,l)$.
\end{proof}
\begin{Remark}
\label{rem:arg}
Note that, as a consequence of $\arg \omega_c \in (0,\pi/2)$, one has
$\arg G(\omega_c) \in (\pi/2,\pi)$. An easy way to see this is to
observe that $\log G(\omega)$ is analytic in the open quadrant
$ \mathbb{R}_{>0} \times \mathrm{i} \mathbb{R}_{>0}$, so that its
imaginary part $\arg(G(\omega))$ is harmonic. On the other hand, along
the positive real axis $\arg G(\omega)=\pi$, while along the positive
imaginary axis,  $\arg G(\omega)\in[\pi/2,\pi]$, as follows from Remark \ref{rem:zeta}. Given that at infinity
$G(\omega)\sim -c/\omega$ whose argument is also in $[\pi/2,\pi]$ and
since $\omega_c$ is in the interior of the quadrant, we deduce the
claim.  A similar argument gives
$\arg G(\omega_c^{-1}) \in (\pi,3\pi/2)$.
\end{Remark}

The following result reduces by symmetry the case of $R_1,R_2$ not both smaller than $1$ to the case $R_1,R_2<1$:
\begin{Lemma}\label{lem:allslopes} Let $r\in\mathcal B$ and  
  assume $x=(x_1,x_2)\in\mathtt{B}_{\e_1}$,
  $y=(y_1,y_2)\in\mathtt{W}_{\e_2}$ with $\e_1, \e_2\in \{0,1\}$.  If
  $R_1=\sqrt{r_1/r_2}\ne1$ and $R_2 = 1/\sqrt{r_1 r_2}\ne1$.  Then,
	\begin{equation}
\begin{split}
		\mathbb{K}^{-1}_{r_1,r_2} (x,y) &= -\mathrm{i}^{1+h(\e_1,\e_2)} \Bigg( a^{\e_2} \tilde{D}_{\tilde{\gamma}_1, \tilde{\gamma}_2} (
(1-2\delta_2) k_1,  (1-2 \delta_1)l_1)) \\& + a^{1-\e_2} \tilde{D}_{\tilde{\gamma}_1,\tilde{\gamma}_2} ((1-2\delta_2) k_2,  (1-2 \delta_1)l_2)   \Bigg)
\end{split}
	\end{equation}
	where  for $i \in\{0,1\}$
	\begin{equation}
		\tilde{\gamma}_{i} = \left\{ \begin{array}{ll}
			\gamma_{R_i} & \mbox{if $R_i<1$} \\
		\gamma_{1/R_i} & \mbox{if $R_i>1$} \end{array}\right. 
	\end{equation}
and 
	\begin{equation}
\label{deltai}
		\delta_{i} = \left\{ \begin{array}{ll}
			0 & \mbox{if $R_i<1$} \\
		1 & \mbox{if $R_i>1$}.\end{array}\right. 
	\end{equation}

\end{Lemma}
\begin{proof}
This is immediate from making the change of variables $u_1 \mapsto u_1^{-1}$ or $u_2 \mapsto u_2^{-1}$ depending on whether $R_1>1$ or $R_2>1$ (or both) in the formula for $\mathbb{K}^{-1}_{r_1,r_2}$ given in Lemma~\ref{CJ16Lem4.3}. 

\end{proof}

\begin{Remark}
  Expanding the formulas for the slopes given
  in~\eqref{eq:rho1},~\eqref{eq:rho2} in terms of their integral
  formulas given in Lemma~\ref{CJ16Lem4.3}, shows that $R_1=1$ or
  $R_2=1$ is equivalent to $|\rho_1|=|\rho_2|$; see
  Lemma~\ref{lem:deltasandrhos} below. For the remaining computations
  for the speed, it is sufficient to consider $R_1<1$ and $R_2<1$
  (recall that we already computed the speed for $|\rho_1|=|\rho_2|$
  at the beginning of this section) and use symmetry.

\end{Remark}

\subsection{Computation of the speed of growth}

We first express the speed as a function of $\omega_c$, $\delta_1$ and
$\delta_2$, which are themselves functions of $R_1,R_2$ and therefore
of $r_1,r_2$, see Lemma \ref{lem:speed}.  Then, we find a formula
(Lemma \ref{lem:slopes}) for the slope $\rho$ also in terms of
$\omega_c$, $\delta_1$ and $\delta_2$.  Finally, in Lemma
\ref{lem:speed2} we put together the two results to obtain the speed
as function of the slope and formula \eqref{eq:v}.

\begin{Lemma}\label{lem:speed} For $r\in\mathcal B$, let $v=v(r)$ be defined as
  \begin{eqnarray}
    \label{eq:nu}
   v:= v(r)=\pi_{\rho(r)}(e_1\in\eta)-\pi_{\rho(r)}(e_2\in\eta)
  \end{eqnarray}
  (compare with Eqs. \eqref{eq:vrho} and \eqref{eq:c1c2}).  For
  $R_1,R_2, \delta_1,$ and $\delta_2$ as given in
  Lemma~\ref{lem:allslopes} with $R_1 , R_2 \not = 1$, we have 
\begin{equation}
v= \frac{(-1)^{\delta_1}}{\pi} ( \arg G(\omega_c) -\pi ) - \frac{(-1)^{\delta_1+\delta_2}}{\pi} \arg \omega_c + {\bf 1}_{(\delta_1,\delta_2)=(0,0)} -{\bf 1}_{(\delta_1,\delta_2)=(1,0)}.
\end{equation}
\end{Lemma}
Here we have defined $v$ to be a function of $r$
whereas in \eqref{eq:vrho} we have defined it as a function of $\rho$. It turns
out that this is equivalent because  the correspondence
$r\in\mathcal B\leftrightarrow \rho(r)\in\mathcal L$ is a
bijection~\cite{KOS03}, as already mentioned.
\begin{proof}
  From \eqref{eq:statistics} and recalling that vertices $w_0,b_0,w_1$ are in the same fundamental domain,
we have
\begin{equation}
v =  a\mathrm{i} \mathbb{K}_{r_1,r_2}^{-1} ((1,0),(0,1)) - a \mathbb{K}^{-1}_{r_1,r_2}((1,0),(2,1)).
\end{equation}
We apply the formula in Lemma~\ref{lem:allslopes} which gives 
\begin{equation}
\begin{split}
&v= a \mathrm{i}(-\mathrm{i}) ( \tilde{D}_{\tilde{\gamma}_1, \tilde{\gamma}_2}(2\delta_2-1,2\delta_1-1)+a  \tilde{D}_{\tilde{\gamma}_1, \tilde{\gamma}_2}(0,0) )\\
&\,\,\,	-a(a\tilde{D}_{\tilde{\gamma}_1, \tilde{\gamma}_2}(0,0)+\tilde{D}_{\tilde{\gamma}_1, \tilde{\gamma}_2}(2\delta_2-1,1-2\delta_1))\\
&= a [\tilde{D}_{\tilde{\gamma}_1, \tilde{\gamma}_2}(2\delta_2-1,2\delta_1-1)- \tilde{D}_{\tilde{\gamma}_1, \tilde{\gamma}_2}(2\delta_2-1,1-2\delta_1)] \\
&= a (\tilde{S}(2\delta_2-1,2\delta_1-1)- \tilde{S}(2\delta_2-1,1-2\delta_1))+
\frac{c}{4\pi \mathrm{i}} \left( \int_{\overline{\omega}_c}^{\omega_c} + \int_{- \overline{\omega}_c}^{-\omega_c} \right) \frac{d\omega}{\omega} \\
&\times \frac{\mathrm{i}^{(1-2\delta_2)+(1-2\delta_1)} G(\omega^{-1})^{-1+2\delta_1} G(\omega)^{-1+2 \delta_2} -\mathrm{i}^{1-2 \delta_2-(1-2 \delta_1)} G(\omega^{-1})^{1-2\delta_1} G(\omega)^{-1+2\delta_2} }{\sqrt{\omega^2+2c} \sqrt{\omega^{-2}+2c}} \\
\end{split}
\end{equation}
where we have used the formula from Lemma~\ref{lem:existsomega}.   The formulas for $\tilde{S}$ can be evaluated as given in Lemma~\ref{lem:existsomega}.  Using 
\begin{eqnarray}
  \label{eq:G-1}
  \frac1{G(\omega)}=-\frac1{\sqrt{2c}}(\omega+\sqrt{\omega^2+2c})  
\end{eqnarray}
together with 
	$G(\omega^{-1})+G(\omega^{-1})^{-1} = -\sqrt{2/c} \sqrt{\omega^{-2}+2c}$, we can simplify the integrand above and we obtain 
\begin{equation}
\begin{split}
v&=  {\bf 1}_{(\delta_1,\delta_2)=(0,0)}-  {\bf 1}_{(\delta_1,\delta_2)=(1,0)}+ \frac{c}{4\pi \mathrm{i}} \left( \int_{\overline{\omega}_c}^{\omega_c} + \int_{- \overline{\omega}_c}^{-\omega_c} \right) \frac{d\omega}{\omega} 
\sqrt{\frac{2}{c}} \frac{ G(\omega)^{-1+2\delta_2} (-1)^{\delta_1+\delta_2}}{\sqrt{\omega^2+2c}}  \\
&= {\bf 1}_{(\delta_1,\delta_2)=(0,0)}- {\bf 1}_{(\delta_1,\delta_2)=(1,0)}+ \frac{c}{2\pi \mathrm{i}}  \int_{\overline{\omega}_c}^{\omega_c}  \frac{d\omega}{\omega}  
  \sqrt{\frac{2}{c}} \frac{ G(\omega)^{-1+2\delta_2} (-1)^{\delta_1+\delta_2}}{\sqrt{\omega^2+2c}}  
\label{similar}
\end{split}
\end{equation}
where we used \eqref{eq:G-G}. We can now evaluate the above integral by noting that $\frac{d}{d \omega} \log G(\omega)= -1/\sqrt{\omega^2+2c}$ and \eqref{eq:G-1}, which implies that 
\begin{equation}
\begin{split}
v= \frac{1}{2\pi \mathrm{i}} \left[ (-1)^{\delta_1} \log G(\omega) +(-1)^{\delta_1+\delta_2+1} \log \omega \right]_{\omega=\overline{\omega}_c}^{\omega=\omega_c} +{\bf 1}_{(\delta_1,\delta_2)=(0,0)}- {\bf 1}_{(\delta_1,\delta_2)=(1,0)}.
\end{split}
\end{equation}
By using the fact that $\arg \omega_c \in (0,\pi/2)$ and
$\arg G(\omega_c) \in (\pi/2, \pi)$, see Remark \ref{rem:arg}, and the convention $\arg(\cdot)\in(-\pi/2,3\pi/2]$ we evaluate the above equation
which gives the result.
\end{proof}

The formulas for the two components of the slope $\rho_1$ and $\rho_2$ are functions of the magnetic
coordinates. Using the height function convention, translation
invariance and~\eqref{eq:statistics}, we find that
\begin{equation}\label{eq:rho1}
\begin{split}
\rho_1(r_1,r_2) &=
 a  \left( \mathbb{K}^{-1}_{r_1,r_2} ((1,2),(0,1)) -  \mathbb{K}^{-1}_{r_1,r_2} ((1,0),(2,1)) \right).
\end{split}
\end{equation}
which can be seen by using Fig.~\ref{fig:fundamental}.
Similarly, 
\begin{equation}\label{eq:rho2}
\begin{split}
 \rho_2(r_1,r_2) &=
 a \mathrm{i} \left( \mathbb{K}^{-1}_{r_1,r_2} ((1,2),(2,1)) -  \mathbb{K}^{-1}_{r_1,r_2} ((1,0),(0,1)) \right).
\end{split}
\end{equation}
 We have the following formulas for the slopes $\rho_1, \rho_2$. 

\begin{Lemma}\label{lem:slopes}
For $R_1,R_2, \delta_1,$ and $\delta_2$ as  in Lemma~\ref{lem:allslopes} with $R_1 , R_2 \not = 1$, we have
\begin{equation}
\label{rho1}
\rho_1 = \frac{ (-1)^{\delta_1}}{\pi} \left( \arg G(\omega_c)-\pi \right) + \frac{ (-1)^{\delta_2}}{\pi} \left( \arg G(\omega_c^{-1})-\pi \right)+{\bf 1}_{(\delta_1,\delta_2)=(0,1)} -{\bf 1}_{(\delta_1,\delta_2)=(1,0)}
\end{equation}
and 
\begin{equation}
\label{rho2}
\rho_2 = -\frac{ (-1)^{\delta_1}}{\pi} \left( \arg G(\omega_c)-\pi \right) + \frac{ (-1)^{\delta_2}}{\pi} \left( \arg G(\omega_c^{-1})-\pi \right)+{\bf 1}_{(\delta_1,\delta_2)=(1,1)} -{\bf 1}_{(\delta_1,\delta_2)=(0,0)}.
\end{equation}
\end{Lemma}
\begin{proof}
For $\rho_1$, we have from~\eqref{eq:rho1} and Lemma~\ref{lem:allslopes} that 
\begin{equation}
\begin{split}
\rho_1 &= a    [a \tilde{D}_{\tilde{\gamma}_1,\tilde{\gamma}_2} (0,0)+ \tilde{D}_{\tilde{\gamma}_1,\tilde{\gamma}_2} (1-2\delta_2,2 \delta_1-1) ]\\& - a [ \tilde{D}_{\tilde{\gamma}_1,\tilde{\gamma}_2} (2\delta_2-1,1-2\delta_1)+ a\tilde{D}_{\tilde{\gamma}_1,\tilde{\gamma}_2} (0,0) ]  \\
&=a [ \tilde{D}_{\tilde{\gamma}_1,\tilde{\gamma}_2} (1-2\delta_2,2 \delta_1-1)-\tilde{D}_{\tilde{\gamma}_1,\tilde{\gamma}_2} (2\delta_2-1,1-2\delta_1) ].
\end{split}
\end{equation}
We apply the integral formula given in Lemma~\ref{lem:existsomega} which gives
\begin{equation}
\begin{split}
\rho_1 & = {\bf 1}_{(\delta_1,\delta_2)=(0,1)} - {\bf 1}_{(\delta_1,\delta_2)=(1,0)} + \frac{c}{4\pi \mathrm{i}} \left( \int_{\overline{\omega}_c}^{\omega_c} + \int_{-\overline{\omega}_c}^{-\omega_c} \right) \frac{d \omega}{\omega} \\
& \times (-1)^{\delta_1+\delta_2} \frac{  G(\omega^{-1})^{2 \delta_1-1} G(\omega)^{1-2\delta_2} - G(\omega^{-1})^{1-2\delta_1} G(\omega)^{2\delta_2-1}}{\sqrt{\omega^2+2c} \sqrt{\omega^{-2}+2c}} .
\end{split}
\end{equation}
Expanding out the above integrand and using the formula for $G(\omega)$ 
 gives
\begin{equation}
\begin{split}
\rho_1 & = {\bf 1}_{(\delta_1,\delta_2)=(0,1)} - {\bf 1}_{(\delta_1,\delta_2)=(1,0)}\\ & + \frac{1}{4\pi \mathrm{i}}  \left( \int_{\overline{\omega}_c}^{\omega_c} + \int_{-\overline{\omega}_c}^{-\omega_c} \right) {d \omega} 
\left[ \frac{(-1)^{\delta_2}}{\omega^2\sqrt{\omega^{-2}+2c}} -\frac{(-1)^{\delta_1}}{\sqrt{\omega^{2}+2c}}\right] .
\end{split}
\end{equation}
The above integrals can be computed  similarly to \eqref{similar}, with the result \eqref{rho1}.

For $\rho_2$, we have from~\eqref{eq:rho2} and Lemma~\ref{lem:allslopes} that 
\begin{equation}
\begin{split}
\rho_2 &= a \mathrm{i}  (-\mathrm{i}) (a \tilde{D}_{\tilde{\gamma}_1,\tilde{\gamma}_2} (0,0)+ \tilde{D}_{\tilde{\gamma}_1,\tilde{\gamma}_2} (1-2\delta_2,1-2 \delta_1) )\\
& -a \mathrm{i} (-\mathrm{i})  ( \tilde{D}_{\tilde{\gamma}_1,\tilde{\gamma}_2} (2\delta_2-1,2\delta_1-1)+ a\tilde{D}_{\tilde{\gamma}_1,\tilde{\gamma}_2} (0,0) )   \\
&=a ( \tilde{D}_{\tilde{\gamma}_1,\tilde{\gamma}_2} (1-2\delta_2,1-2 \delta_1)-\tilde{D}_{\tilde{\gamma}_1,\tilde{\gamma}_2} (2\delta_2-1,2\delta_1-1) ).
\end{split}
\end{equation}
We apply the integral formula given in Lemma~\ref{lem:existsomega} which gives
\begin{equation}
\begin{split}
  \rho_2 & = {\bf 1}_{(\delta_1,\delta_2)=(1,1)} - {\bf 1}_{(\delta_1,\delta_2)=(0,0)} + \frac{c}{4\pi \mathrm{i}} \left( \int_{\overline{\omega}_c}^{\omega_c} + \int_{-\overline{\omega}_c}^{-\omega_c} \right) \frac{d \omega}{\omega} \\
  & \times \frac{ \mathrm{i}^{2\delta_2+2\delta_1+2}
    G(\omega^{-1})^{1-2 \delta_1} G(\omega)^{1-2\delta_2} -
    \mathrm{i}^{2-2\delta_1-2 \delta_2}
    G(\omega^{-1})^{-(1-2\delta_1)}
    G(\omega)^{-(1-2\delta_2)}}{\sqrt{\omega^2+2c}
    \sqrt{\omega^{-2}+2c}} .
\end{split}
\end{equation}
Expanding out the above integrand and using the formula for $G(\omega)$ gives
\begin{equation}
\begin{split}
\rho_2 & = {\bf 1}_{(\delta_1,\delta_2)=(1,1)} - {\bf 1}_{(\delta_1,\delta_2)=(0,0)}\\& + \frac{1}{4\pi \mathrm{i}}  \left( \int_{\overline{\omega}_c}^{\omega_c} + \int_{-\overline{\omega}_c}^{-\omega_c} \right) {d \omega}  \left[\frac{(-1)^{\delta_1}}{\sqrt{\omega^2+2c}} +\frac{(-1)^{\delta_2}}{\omega^2\sqrt{\omega^{-2}+2c}}\right] .
\end{split}
\end{equation}
Computing the integrals, one gets \eqref{rho2}.
\end{proof}

It follows from Lemma \ref{lem:slopes} that
\begin{equation}
\label{eq:G+}
\rho_1+\rho_2=\frac2\pi(-1)^{\delta_2}\arg G(\omega_c^{-1})-3(-1)^{\delta_2}
\end{equation}
and 
\begin{equation}
\label{eq:G-}
\rho_1-\rho_2=\frac2\pi(-1)^{\delta_1}\arg G(\omega_c)-(-1)^{\delta_1}.
\end{equation}
From these equations and Remark \ref{rem:arg} we see that
\begin{Lemma} \label{lem:deltasandrhos}
Let $r\in\mathcal B$ with $R_1,R_2\ne 1$, $\rho=\rho(r)$ and $\delta_1,\delta_2$ be defined as in \eqref{deltai}. Recall that $\rho\in\mathcal L$, with $\mathcal L$ defined through \eqref{eq:calL}.
We have
\begin{equation}
  \begin{array}{llr}
\label{eq:deltasandrhos}
\delta_1=0 &\mbox{ if and only if } & 0<\rho_1-\rho_2<1,\\
 \delta_1=1 &\mbox{ if and only if } &-1<\rho_1-\rho_2<0,\\
\delta_2=0  &\mbox{ if and only if } &-1<\rho_1+\rho_2<0\\
\delta_2=1  &\mbox{ if and only if } &0<\rho_1+\rho_2<1,
  \end{array}
\end{equation}	
while $|\rho_1|=|\rho_2| $ correspond to having either $R_1=1$ or $R_2=1$.  
\end{Lemma}

Putting together Lemmas \ref{lem:speed} and \ref{lem:slopes} we obtain.
\begin{Lemma} \label{lem:speed2}
For $\rho\in\mathcal L$, $|\rho_1|\ne |\rho_2|$, the speed of growth is given by 
\begin{equation}
\label{viro}
v(\rho)= \frac{\rho_1-\rho_2}{2} +\frac{(-1)^{\delta_1+\delta_2}}{2} - \frac{(-1)^{\delta_1+\delta_2}}{\pi}\arg \omega_c
\end{equation}
where $\arg \omega_c\in(0,\pi/2)$ satisfies 
\begin{equation}
\begin{split}
\label{eq:eqoc}
&c\,(-1)^{\delta_1+\delta_2}\sin ( \pi (\rho_1+\rho_2))\sin ( \pi (\rho_1-\rho_2))\\
&=\left( \left(\cos ( \pi(\rho_1-\rho_2))+\cos (2\arg \omega_c) \right) \left(\cos ( \pi (\rho_1+\rho_2))+\cos (2\arg \omega_c) \right)\right)^{\frac{1}{2}}.
\end{split}
\end{equation}
\end{Lemma}

\begin{proof}
Equation \eqref{viro} follows immediately from Lemmas \ref{lem:slopes} and \ref{lem:speed},
so we only need to determine a formula for $\arg \omega_c$.  Let
$\bar X^{1/2}= |G(\omega_c)|$ and $\bar Y^{1/2}=|G(\omega_c^{-1})|$. Then, we
have the following
\begin{equation}\label{eq:XandG}
\sqrt{\frac{{2}}{c}} \omega_c = G(\omega_c) -G(\omega_c)^{-1} = \bar X^{1/2} e^{\mathrm{i} \arg G(\omega_c)} -\bar X^{-1/2} e^{-\mathrm{i} \arg G(\omega_c)}.
\end{equation}
This equation represents a triangle in the upper half plane since $\arg \omega_c \in (0,\pi/2)$,  $\arg G(\omega_c) \in (\pi/2,\pi)$ and $\arg(-e^{-\mathrm{i} \arg G(\omega_c)})=\pi-\arg G(\omega_c) \in (0,\pi/2)$.  By using this triangle, the sine rule gives that
\begin{equation}\label{eq:Xonly}
	\bar X= - \frac{ \sin (\arg (G(\omega_c)+\arg \omega_c)}{ \sin (\arg (G(\omega_c)-\arg \omega_c)}. 
\end{equation}
Similarly,
\begin{equation}\label{eq:YandG}
\sqrt{\frac{{2}}{c}} \omega_c^{-1} = \bar Y^{1/2} e^{\mathrm{i} \arg G(\omega_c^{-1})} -\bar Y^{-1/2} e^{-\mathrm{i} \arg G(\omega_c^{-1})}.
\end{equation}
This equation represents a triangle in the lower half plane since $\arg \omega_c^{-1} \in (-\pi/2,0)$,  $\arg G(\omega_c^{-1}) \in (\pi,3\pi/2)$ and $\arg(-e^{-\mathrm{i} \arg G(\omega_c^{-1})})=\pi-\arg G(\omega_c^{-1}) \in (-\pi/2,0)$.
By using this triangle, the sine rule gives that
\begin{equation}\label{eq:Yonly}
	\bar Y= - \frac{ \sin (\arg (G(\omega_c^{-1})-\arg \omega_c)}{ \sin (\arg (G(\omega_c^{-1})+\arg \omega_c)}. 
\end{equation}
We also have using~\eqref{eq:XandG} and~\eqref{eq:YandG}
\begin{equation}
\begin{split}
{\frac{{2}}{c}}&=\sqrt{\frac{{2}}{c}}\omega_c \sqrt{\frac{{2}}{c}}\omega_c^{-1}\\
&=\bigg(\bar X^{\frac{1}{2}} e^{\mathrm{i} \arg G(\omega_c)} -\bar X^{-\frac{1}{2}} e^{-\mathrm{i} \arg G(\omega_c)}\bigg) \bigg(\bar Y^{\frac{1}{2}} e^{\mathrm{i} \arg G(\omega_c^{-1})} -\bar Y^{-\frac{1}{2}} e^{-\mathrm{i} \arg G(\omega_c^{-1})}\bigg).
\end{split}
\end{equation}
Taking  the real part of the above equation yields 
\begin{equation}
\begin{split}
&(\bar X\bar Y+1) \cos( \arg G(\omega_c) +\arg G(\omega_c^{-1})) - (\bar X+\bar Y) \cos ( \arg G(\omega_c) -\arg G(\omega_c^{-1}))\\
 &=\frac2c\sqrt{\bar X\bar Y}.
\end{split}
\end{equation}
Using the equations for~\eqref{eq:Xonly} and~\eqref{eq:Yonly} we find that 
\begin{equation}
\begin{split}
&\frac{c}{2}\sin(2 \arg G(\omega_c)) \sin(2 \arg G(\omega_c^{-1})) = \bigg( \sin(\arg G(\omega_c)+\arg \omega_c) \\
	&\times  \sin(\arg G(\omega_c^{-1})+\arg \omega_c) \sin(\arg G(\omega_c)-\arg \omega_c)\sin(\arg G(\omega_c^{-1})-\arg \omega_c) \bigg)^{1/2}.
\end{split}
\end{equation}
The equations \eqref{eq:G+}, \eqref{eq:G-} for  $\arg G(\omega_c)$ and $\arg G(\omega_c^{-1})$, together with 
angle addition formulas followed by factorizing and simplifying, give the result \eqref{eq:eqoc}.
\end{proof}

Finally, we show how to choose the correct solution of \eqref{eq:eqoc} for $\arg(\omega_c)$:
\begin{proof}[Proof of Theorem \ref{th:v}]
  As we remarked at the beginning of Section \ref{sec:compv}, we can
  assume that $|\rho_+|\ne |\rho_-|$. Also, given symmetries \eqref{eq:simmv}, \eqref{eq:simmv1}, we can assume that $0<\rho_1-\rho_2<1$ and $-1<\rho_1+\rho_2<0$
(i.e. $\delta_1=\delta_2=0$) and all other cases immediately follow.

Solving \eqref{eq:eqoc} for $\arg(\omega_c)$, we find that 
\begin{eqnarray}
\label{eq:eqoc2}
&  \arg(\omega_c)=\frac12\arccos(y_\pm),
\\
&y_\pm=\frac{-\cos(\rho_-)-\cos(\rho_+)\pm \sqrt{(\cos(\rho_-)-\cos(\rho_+))^2+4 c^2 \sin^2(\rho_-)\sin^2(\rho_+)^2}}2
\end{eqnarray}
and the r.h.s. of \eqref{eq:eqoc2} must be in $[0,\pi/2]$ thanks to Lemma \ref{lem:existsomega}.
Imposing that the speed vanishes at $\rho=(-1/2,-1/2)$, as a consequence of \eqref{eq:simmv1}, imposes that the correct solution is $y_-$ and \eqref{eq:v} is proven.
\end{proof}

\subsection{Asymptotic expansion of the speed of growth for small slope}
\label{sec:propv}
The proof of Theorem \ref{th:a} is a lengthy but straightforward
application of calculus; we do not give details. The explicit form of
$f_1(r)$ and $f_2(r)$ turn out to be
\begin{eqnarray}
    \label{eq:f1f2}
    f_1(r)=\frac14\left(1 + r^2 - \sqrt{1 + 2 (-1 + 8 c^2) r^2 + r^4}\right)\\
    f_2(r)=\frac1{48}\left (-1 - r^4 + \frac{(1 + r^2) (1 + (-2 + 32 c^2) r^2 + r^4)}{\sqrt{
   1 + 2 (-1 + 8 c^2) r^2 + r^4}}\right).
  \end{eqnarray}
  Note that $f_1(\cdot)>0$ whenever $a<1$, i.e. $c<1/2$ (it vanishes identically if
  $c\to1/2$).  To prove the last claim of Theorem \ref{th:a}, we remark
  that the error term $O(\rho_+^5)$ in \eqref{eq:asympt} gives a
  $o(1)$ contribution to the determinant and the trace of $H_\rho$ as
  $\rho\to0$.  Starting from \eqref{eq:asympt}, a Taylor expansion
  shows (with the same notations $r=\rho_-/\rho_+$ as in Theorem
  \ref{th:a}) that
  \begin{eqnarray}
    \label{eq:trH}
    \Tr[H_\rho]=\frac{1+r^2}{2\pi \rho_+}\frac{ f'_1(r)^2 -2f_1(r)f''_1(r)}{2\sqrt2 f_1(r)^{3/2}}+o(1)
  \end{eqnarray}
  as $\rho\to0$.  With similar computations, one finds
  \begin{eqnarray}
    \label{eq:dethas}
    \det(H_\rho)=\frac{(f_1(r)^2+6f_2(r))(2f_1(r)f''_1(r)-f_1'(r)^2)}{16 \pi^2 f_1(r)^2}+o(1).
  \end{eqnarray}
  It is not hard to check that $(f_1(r)^2+6f_2(r))>0$, while
  $2f_1(r)f''_1(r)-f_1'(r)^2$ is negative whenever $a<1$.  As a
  consequence, for $a<1$ the trace of $H_\rho$ diverges and its
  determinant tends to a finite negative limit as $\rho\to0$, as
  claimed.

\section{AKPZ signature of the growth model}
\label{sec:AKPZsig}

In this section we prove Theorem \ref{th:harmonic} and \ref{th:AKPZ}.

\begin{proof}[Proof of Theorem \ref{th:harmonic}]
Note that, as a consequence of $\arg z \in (0,\pi/2)$ for $z\in\mathcal Q^+$, one has
$\arg G(z) \in (\pi/2,\pi)$. An easy way to see this is to
remark that $\log G(z)$ is analytic in $\mathcal Q^+$, so that its
imaginary part $\arg(G(z))$ is harmonic. On the other hand, along
the positive real axis $\arg G(z)=\pi$, while along the positive
imaginary axis, recalling definition \eqref{eq:sqrt}, $\arg G(z)$
is easily seen to be in $[\pi/2,\pi]$. Given that at infinity
$G(z)\sim -c/z$ whose argument is also in $[\pi/2,\pi]$ and
since $\mathcal Q^+$ is open, we deduce the
claim.  A similar argument gives
$\arg G(1/z) \in (\pi,3\pi/2)$. Altogether, we see that $(X(z),Y(z))\in(\pi/2,\pi)\times(\pi,3\pi/2)$.

Next, we check that the map from $z$ to $(X,Y)$ is a local diffeomorphism. 
Writing 
as usual $z=|z| e^{i\theta}$, we
 have (using the usual identities  $\partial_\theta z=i z$ and $\partial_{|z|}z=z/|z|$, as well as the definition $X=\Im\log G(z)$)
 \begin{eqnarray}
   \label{eq:jacob}
   \partial_\theta X=\Re(V(z)),\qquad
\partial_\theta Y=-\Re(V(1/z)),\\
\partial_{|z|}X=\frac1{|z|}\Im(V(z)),\qquad
\partial_{|z|}Y=-\frac1{|z|}\Im(V(1/z))
 \end{eqnarray}
where $ V(z):=\frac{ z G'(z)}{G(z)}$. The determinant of the Jacobian vanishes only if
\begin{eqnarray}
\label{eq:j2}
  \Re(V(z))\Im(V(1/z))=\Re(V(1/z))\Im(V(z))
\end{eqnarray}
and we show in a moment that this condition is never satisfied.
We will see later that  $\Re(V(z))$,  $\Re(V(1/z))$ are non-zero in the whole $\mathcal Q^+$, so \eqref{eq:j2}
is equivalent to 
\begin{eqnarray}
 \frac{\Im(V(z))}{\Re(V(z))}= \frac{\Im(V(1/z))}{\Re(V(1/z))},
\end{eqnarray}
i.e.  $V(z)$ has the same argument as $V(1/z)$, i.e. 
\begin{eqnarray}
\label{eq:primorap}
  \frac{V(z)}{V(1/z)}=\alpha
\end{eqnarray}
for some $\alpha>0$. Now we use the explicit form of $G(z)$, which implies that
\begin{eqnarray}
  \label{eq:V}
V(z)=-\frac z{\sqrt{z^2+2c}}.
\end{eqnarray}
  The l.h.s. of
\eqref{eq:primorap} is real on the positive real axis. One can check
from Remark \ref{rem:zeta} that its
argument is in $[0,\pi/2]$ when the imaginary axis is approached from
the first quadrant. To be precise, if $z$ tends to $iy,y>0$, the
argument tends to $\pi/2$ if $y>1/\sqrt{2c}$ or $y<\sqrt{2c}$ and to
$0$ if $\sqrt{2c}<y<1/\sqrt{2c}$; if $y=\sqrt{2c}$ or $1/\sqrt{2c}$
the limit need not exist but the limit points are in
$[0,\pi/2]$. Also, for $|z|\to\infty$ or $|z|\to 0$ (with $z\in\mathcal Q^+$) the
argument of the l.h.s. of
\eqref{eq:primorap} is in $(0,\pi/2)$ because the ratio is $z\sqrt{2c}+o(z)$ and $z/\sqrt{2c}+o(z)$ respectively. Since
$\arg(V(z)/V(1/z))$ is a harmonic function in the open set $\mathcal Q^+$ with boundary values in $[0,\pi/2]$ and not identically $0$, it vanishes nowhere in $\mathcal Q^+$, so \eqref{eq:primorap} cannot hold for any
$\alpha>0$. By the way, the same harmonicity argument gives that the
real parts of $V(z)$ or $V(1/z)$  vanish nowhere in $\mathcal Q^+$, as
mentioned above.

Once we know that the map $z\mapsto (X,Y)$ is a local diffeomorphism,
we can conclude that it is a global diffeomorphism, via a theorem by
Hadamard \cite{hadamard1906transformations,gordon1972diffeomorphisms}, if we prove that it is proper, i.e. for
every sequence $z_n\in\mathcal Q^+$ such that
$z_n\to z\in\partial \mathcal Q^+$,  $(X(z_n),Y(z_n))$ tends to
a point on the boundary of
$(\pi/2,\pi)\times(\pi,3\pi/2)$\footnote{Hadamard's theorem is
  formulated as a necessary and sufficient condition for a smooth map
  from $\mathbb R^n$ to $\mathbb R^n$ to be a diffeomorphism.  In our
  case the map is from $\mathcal Q^+$ to
  $(\pi/2,\pi)\times(\pi,3\pi/2)$ but the proof works essentially the
  same, see for instance \cite[Th. 4.4]{laslier2015lozenge} for an analogous
  case where the map is between two open connected subsets of
  $\mathbb R^4$}.  Here, when we write $z_n\to z\in\partial \mathcal Q^+$, we
mean that either $z$ is purely real, or it is purely imaginary, or
$|z_n|$ tends to infinity.
When $z_n$ tends to a real limit, it is clear that $\arg G(z)$ tends to $\pi$ while when $|z_n|\to\infty$, $\arg(G(1/z))$ tends to $\pi$.
When instead $z_n$ tends to $iy$, then Remark \ref{rem:zeta} shows that $\arg G(z_n)$ tends to $\pi/2$ if $y>\sqrt{2c}$ while $\arg(G(1/z_n))$ tends to $3\pi/2$ if $y<1/\sqrt{2c}$. Since $\sqrt{2c}\le 1/\sqrt{2c}$, we proved our claim that in all cases $(X(z_n),Y(z_n))$ approaches the boundary of $(\pi/2,\pi)\times(\pi,3\pi/2)$.

The proof of \eqref{eq:v2} follows from Eqs. \eqref{eq:G+},
\eqref{eq:G-} and \eqref{viro}, because for $\rho_+>0,\rho_->0$ one
has $\delta_1=0,\delta_2=1$, cf. \eqref{eq:deltasandrhos}.

  \end{proof}

  \begin{proof}[Proof of Theorem \ref{th:AKPZ}]
    For $a=1$, an explicit computation starting from \eqref{eq:v} shows that 
 \begin{eqnarray}
    \label{eq:detha1}
    \left.    \det(H_\rho)\right|_{a=1}=-\frac{4\pi^2\cos(\rho_+/2)^2\cos(\rho_-/2)^2}{(3+\cos(\rho_+)+\cos(\rho_-)-\cos(\rho_+)\cos(\rho_-))^2}
  \end{eqnarray}
  that is strictly negative in $\mathcal L$.
  Similarly, for $\rho_+>0$
a lengthy but straightforward analysis of \eqref{eq:v} gives
 \begin{eqnarray}
 \left.\det(H_\rho)\right|_{\rho_-=0}=-\frac{\pi^2c^4\sin(\rho_+)^2}{1-2c^2-2c^2\cos(\rho_+)}<0.
 \end{eqnarray}

Therefore, from now on we will assume by symmetry
$0< \rho_+,\rho_-<\pi$.  Since the relation \eqref{eq:XY} between
$(\rho_1,\rho_2)$ and $(X,Y)$ is affine, it is sufficient to prove
that the Hessian of $\arg z(X,Y)$ has a strictly negative
determinant. We know  from Theorem \ref{th:a}  that $\det(H_\rho)<0$ for $\rho$ small, so we need only
to prove that the sign of the determinant vanishes nowhere.

We have from \eqref{eq:jacob}
\begin{eqnarray}
 \frac{\partial\arg(z)}{\partial X} =\frac1{\Re(V(z))}, \qquad  \frac{\partial\arg(z)}{\partial Y}=-\frac1{\Re(V(1/z))}.
\end{eqnarray}
Differentiating once more,
\begin{eqnarray}
 & \frac{\partial^2\arg(z)}{\partial X^2}=-\frac1{[\Re(V(z))]^2}\partial_X \Re(V(z))\\
 & \frac{\partial^2\arg(z)}{\partial X\partial Y}=-\frac1{[\Re(V(z))]^2}\partial_Y \Re(V(z))=\frac1{[\Re(V(1/z))]^2}\partial_X \Re(V(1/z)) \\
 & \frac{\partial^2\arg(z)}{\partial Y^2}=\frac1{[\Re(V(1/z))]^2}\partial_Y \Re(V(1/z)).
\end{eqnarray}
Then, the determinant of the Hessian of $\theta$ as function of $(X,Y)$ is
\begin{eqnarray}
  -\frac{\partial_X \Re(V(z))\partial_Y \Re(V(1/z))- \partial_Y \Re(V(z))\partial_X \Re(V(1/z))}{[\Re(V(1/z))]^2[\Re(V(z))]^2}
\end{eqnarray}
and it is enough to prove that the numerator vanishes nowhere.
A few lines of computations (writing $\partial_X=1/(\partial_\theta X)\partial_\theta+1/(\partial_{|z|} X)\partial_{|z|}$ and similarly for $\partial_Y$) show that the numerator equals
\begin{multline}
\label{eq:combinaz1}
  |z|\left(\frac1{\Im(V(z))\Re(V(1/z))}-\frac1{\Re(V(z))\Im(V(1/z))}
\right)\\
\times\left(\partial_\theta \Re(V(z))\partial_{|z|}\Re(V(1/z))-\partial_\theta \Re(V(1/z))\partial_{|z|}\Re(V(z))
\right).
\end{multline}
Again, a straightforward  computation gives that the second line equals
\begin{eqnarray}
  \frac1{|z|}\left(\Im(U(z))\Re(U(1/z))-\Re(U(z))\Im(U(1/z))\right)
\label{eq:combinaz2}
\end{eqnarray}
where
\begin{eqnarray}
  U(z)=z\frac{G'(z)+z G''(z)}{G(z)}-V(z)^2= -\frac{2cz}{(z^2+2c)^{3/2}}.
\end{eqnarray}
The first line in \eqref{eq:combinaz1} never vanishes, because we already proved that 
\eqref{eq:j2} is never satisfied. Similarly, \eqref{eq:combinaz2} vanishes only if 
\begin{eqnarray}
  \label{eq:secondorap}
\frac{U(z)}{U(1/z)}=\alpha
\end{eqnarray}
for some $\alpha>0$. 
The  l.h.s. of \eqref{eq:secondorap} is positive on the positive
real axis. This time, one sees that  ${\arg}(U(z)/U(1/z))\in (-\pi/2,0)$ for
$|z|\to\infty$ and in ${\arg}(U(z)/U(1/z))\in [-\pi/2,0]$ (and not identically zero) when the positive imaginary axis is
approached. Since ${\arg} (U(z)/U(1/z))$ is harmonic on the open set $\mathcal Q^+$, it vanishes nowhere and \eqref{eq:secondorap} is nowhere satisfied.
  \end{proof}

  \section*{Acknowledgements} We are  grateful to Patrik Ferrari and Sanjay Ramassamy for enlightening discussions. We would also like to thank the referee for their careful reading of our manuscript. 
   F.T. was partially supported by the CNRS PICS grant ``Interfaces al\'eatoires discr\`etes et dynamiques de Glauber'',  by  the  ANR-15-CE40-0020-03
   Grant LSD and by  Labex MiLyon (ANR-10-LABX-0070).
   \bibliographystyle{plain}
\bibliography{Biblio}

\end{document}